\documentclass[11pt]{amsart}

\usepackage{amsmath, amssymb, amsthm, xcolor,mathtools}
\usepackage{verbatim}
\usepackage{graphicx}

\usepackage[foot]{amsaddr}
\usepackage[toc]{appendix} 
\usepackage{latexsym}
\usepackage{amsfonts}
\usepackage{amssymb}
\usepackage{amsmath}
\usepackage{mathrsfs}
\usepackage{bbm}
\usepackage{dsfont}
\usepackage{hyperref}
\allowdisplaybreaks

\usepackage[toc]{appendix}

\newcommand{\rn}{\mathbb{R}^{\mathbb{N}}}
\newcommand{\rpn}{\mathbb{R}^{\mathbb{N}}_{+}}
\newcommand{\ut}{\underline{t}}
\newcommand{\be}{\begin{equation}}
\newcommand{\ee}{\end{equation}}
\newcommand{\bea}{\begin{eqnarray}}
\newcommand{\eea}{\end{eqnarray}}
\newcommand{\bean}{\begin{eqnarray*}}
\newcommand{\eean}{\end{eqnarray*}}
\newcommand{\brray}{\begin{array}}
\newcommand{\erray}{\end{array}}

\newcommand{\bdefinition}{\begin{definition}\rm}
\newcommand{\bthm}{\begin{thm}}
\newcommand{\blmma}{\begin{lmma}}
\newcommand{\bppsn}{\begin{ppsn}}
\newcommand{\bcrlre}{\begin{crlre}}
\newcommand{\bxmpl}{\begin{xmpl}}
\newcommand{\brmrk}{\begin{rmrk}\rm}

\newcommand{\edefinition}{\end{definition}}
\newcommand{\ethm}{\end{thm}}
\newcommand{\elmma}{\end{lmma}}
\newcommand{\eppsn}{\end{ppsn}}
\newcommand{\ecrlre}{\end{crlre}}
\newcommand{\exmpl}{\end{xmpl}}
\newcommand{\ermrk}{\end{rmrk}}

\newcommand{\bbc}{\mathbb{C}}
\newcommand{\bbz}{\mathbb{Z}}

\newcommand{\bbn}{\mathbb{N}}
\newcommand{\bbr}{\mathbb{R}}

\newcommand{\rp}{\mathbb{R}_{+}}

\newcommand{\rpinf}{\mathbb{R}_{+}^{\infty}}

\newcommand{\cla}{\mathcal{A}}
\newcommand{\clb}{\mathcal{B}}

\newcommand{\clf}{\mathcal{F}}
\newcommand{\clh}{\mathcal{H}}

\newcommand{\clk}{\mathcal{K}}
\newcommand{\cll}{\mathcal{L}}

\newcommand{\clr}{\mathcal{R}}

\newcommand{\clg}{\mathcal{G}}

\newcommand{\dplus}{\mathbb{R}^{d}_{+}}

\makeatletter
\let\@wraptoccontribs\wraptoccontribs
\makeatother

\title{Doubly Commuting Semigroups of Isometries}
\author{C.H. Namitha$^{\dagger}$}
\address{ $^{\dagger}$ (Formerly) The Institute of Mathematical Sciences (HBNI), 4th cross street, CIT Campus, Taramani, Chennai, India, 600113}
\email{namithachanguli7@gmail.com}

\author{S. Sundar$^{\ast \dagger}$}
\address{$^{\ast \dagger}$ The Institute of Mathematical Sciences (HBNI), 4th cross street, CIT Campus, Taramani, Chennai, India, 600113}
\email{ sundarsobers@gmail.com}
\author{Shankar Veerabathiran $^{\ast}$}
\address{ $^{\ast}$ Department of Mathematics, SRM University AP, 
Amaravati, Andhra Pradesh, India, 522502}
\email{shankarunom@gmail.com}

\usepackage{enumerate}
\usepackage{bbm}
\usepackage{cleveref}

\newtheorem{definition}{Definition}[section]
\newtheorem{prop}[definition]{Proposition}
\newtheorem{theorem}[definition]{Theorem}

\newtheorem{lemma}[definition]{Lemma}
\newtheorem{remark}[definition]{Remark}
\numberwithin{equation}{section}
\newcommand{\RNum}[1]{\uppercase\expandafter{\romannumeral #1\relax}}

\begin{document}
\maketitle

\begin{abstract}
    In this paper, we discuss the structure of doubly commuting semigroups of isometries. We record a new proof of Cooper's theorem in the Hilbert module setting. We discuss the Fell topology on the set of equivalence classes of irreducible, doubly commuting isometric representations of $\bbr_{+}^{d}$.  We show that if $d$ is finite, the topology is $T_0$. We indicate the pathologies that occur when $d=\infty$. In particular, we show that  Wold decomposition fails for isometric representations of $\bbr_{+}^{\infty}$ and prove that the Fell topology on the set of equivalence classes of irreducible, doubly commuting isometric representations of $\bbr_{+}^{\infty}$ is not $T_0$.
\end{abstract}

\noindent {\bf AMS Classification No. :} {Primary 47D03; Secondary 47A65.}  

 \section{Introduction}
The study of semigroups of isometries (also called isometric representations)  on  Hilbert spaces is an active area of research, and they can be studied from various perspectives leading to a rich interaction between several branches of functional analysis. For example, the role of the Hardy space $H^2(\mathbb {D})$ and the associated complex analysis tools in the study of the shift operator $S$ on $\ell^2(\bbn_0)$ is well known (see \cite{Nagy_Foias_new}). Similary,  the $C^{*}$-algebra associated with $S$, i.e. the Toeplitz algebra $\mathcal{T}$ plays a crucial role in Cuntz' approach (\cite{Cuntz_Bott}) to Bott periodicity in $K$-theory. 
Over the last couple of decades, there have been many papers on $C^{*}$-algebras associated with discrete semigroups (we merely cite \cite{Li_Oberwolfach} and refer the reader to the references therein) and also on its representation theory, or in other words isometric representations of discrete semigroups.  Isometric representations in the continuous case and the associated $C^{*}$-algebras have also been studied by many authors (\cite{Arveson_semigroup}, \cite{Bhat_Skeide},   \cite{Vijay_Kumar}, \cite{Renault_Muhly},  \cite{Piyasa_NYJM},  \cite{Sundar_Ore}). 

 In this paper, we revisit some aspects of isometric representations in the continuous case. We focus only on the doubly commuting semigroups of isometries of  $\bbr_{+}^{d}$ and $\bbr_{+}^{\infty}$. The goal of this paper is to give a self contained exposition to the structure of doubly commuting semigroups of isometries.  We do not take recourse to $C^{*}$-algebras  but adopt  a purely operator theoretic viewpoint. For, we believe that it is interesting to study isometric representations of semigroups that are not locally compact (like the case of $\bbr_{+}^{\infty}$ and $\ell^2_{+}$), where the operator algebraic techniques may not work.  Hence, we think that it will be of some use to record our approach for future reference.

  We start with a new proof of Cooper's theorem in the Hilbert module setting. 
Cooper's theorem is a fundamental result in the theory of semigroups of isometries, and it asserts that if $\{V_t\}_{t \geq 0}$ is a strongly continuous, pure isometric representation of $\bbr_{+}$ on a Hilbert space $\clh$, then $\clh \cong L^{2}(\bbr_{+}) \otimes \cll$ for some Hilbert space $\cll$ and $V_t \cong S_t \otimes 1$, where $\{S_t\}_{t \geq 0}$ is the usual shift semigroup. There are many proofs of this result. The older proofs (\cite{Cooper}, \cite{Masani}) rely on the theory of unbounded operators which is delicate with domain issues. There are also  operator algebraic proofs (\cite{Sundar_Cooper}, \cite{Zacharias}), and this demands building enough machinery (crossed products/groupoids/Stone-von Neumman theorem) a priori to force the proof.  
The proof that we present in this paper relies only on basic operator theory and functional analysis. Another such proof can be found in \cite{Bhat_Skeide}.  We believe that our proof is simple and will be of independent interest. Moreover, our proof is extremely analogous to the proof in the discrete case.

 The main ideas  are explained next. First, let us recall the proof in the discrete case for the structure of a single isometry $V$ on a Hilbert space $\clh$ which is pure, i.e. $V^{*n} \to 0$ in SOT as $n \to \infty$.   The proof has three steps, and they are follows: 
\begin{enumerate}
    \item[\textbf{Step} 1:] Set 
\[
K_V:=\{\xi \in \clh: V^{*n}\xi=0 \textrm{~for all $n \geq 1$}\}.\]
We first observe that $K_V \neq 0$.
\item[\textbf{Step} 2:]  In this step, we show that the map 
\[
U: \ell^2(\bbn_0) \otimes K_V \ni \delta_n \otimes \xi  \to V^n\xi \in \clh
\]
is an isometry by computing inner products on a total set of vectors. Moreover, by definition, it follows that 
\[
U(S \otimes 1)=VU.\]
Also, the range of $U$ is reducing for $\{V^n: n \geq 1\}$. 
\item[\textbf{Step} 3:] Consider the restriction $W:=V|_{(U\clh)^{\perp}}$ which is again a pure isometry by hypothesis. Applying Step 1 to $W$, we get  $0 \neq K_W \subset K_V \subset U\clh$ which is a contradiction unless $(U\clh)^{\perp}=0$. Hence, $U$ is a unitary, and now the proof is over. 

\end{enumerate}

We cannot adapt the above strategy step by step for a $1$-parameter strongly continuous semigroup of isometries $V=\{V_t\}_{t \geq 0}$ as $L_V=0$ in this case as $V$ is strongly continuous. The modification required is that rather than seeking a common eigenvector corresponding to the eigenvalue $0$, we look for a common eigenvector corresponding to a non-zero eigenvalue (or a non-zero character in the continuous case).  Then, the steps are as follows.
Let $V=\{V_t\}_{t \geq 0}$ be a strongly continuous pure (i.e. $V_tV_t^* \to 0$ in SOT as $t \to \infty$) semigroup of isometries on a Hilbert space $\clh$. 
\begin{enumerate}
    \item[(1)] Let $z$ be a complex number such that $Re(z)>0$. Set 
    \[
    K_{z}:=\{\xi \in H: V_{t}^{*}\xi=e^{-zt}\xi \textrm{~for all $t>0$}\}. 
    \]
    We show that $K_z$ is non-zero by appealing to the discrete case, and by an averaging trick. 
    \item[(2)] Let $\{S_t\}_{t \geq 0}$ be the shift semigroup on $L^2((0,\infty))$, and set $f_z(t):=\sqrt{2 Re(z)}e^{-zt}$. Then, it is straightforward to compute inner products on the total set  $\{S_tf_z \otimes \xi: t \geq 0, \xi \in K_z\}$ to deduce that there exists an isometry 
$    U:L^2((0,\infty)) \otimes K_z \ni S_tf_z \otimes  \xi \to V_t\xi \in \clh$ such that 
\[
U(S_t\otimes 1)=V_tU
\]
for every $t \geq 0$. The rest of the arguments are exactly the same as in the discrete case. Note that the role played by `$\delta_0$' in the discrete case is now played by $f_z$ in the continuous case. 
\end{enumerate}
We generalise the above to the doubly commuting case and to the Hilbert module setting.  Here, we consider strictly continuous semigroups as the strict topology is a more appropriate topology for Hilbert modules. The module version creates a few issues at Step 3 but the main idea remains the same. 
Thus, we obtain a  proof of the following theorem; the case $d=1$ is first due to Bhat and Skeide (\cite{Bhat_Skeide}) in the module setting (see also \cite{Sundar_Cooper} for a groupoid proof).

\begin{theorem}
\label{main_intro}
Let $d \geq 1$ be a natural number. Let $\cla$ be a separable $C^{*}$-algebra, and let $E$ be a countably generated Hilbert $\cla$-module. Suppose $V=\{V_{\underline{t}}\}_{\underline{t} \in \bbr_{+}^{d}}$ is a  strictly continuous semigroup of isometries which is doubly commuting, i.e. for $i \neq j$, \[V_{te_i}V_{se_j}^{*}=V_{se_j}^{*}V_{te_i}\] for $s,t \in \bbr_{+}$. Here, $\{e_j\}_{j=1}^{d}$ is the standard basis for $\bbr^d$. Suppose $V$ is strongly pure, i.e. for every $i \in \{1,2,\cdots, d\}$, $\displaystyle \lim_{t \to \infty}V_{te_i}V_{te_i}^{*}=0$ in the strict topology.   Then, there exists a Hilbert $\cla$-module $K$ and a unitary $U:L^2(\bbr_{+}^{d}) \otimes K \to E$ such that 
\[
U(S_{\underline{t}}\otimes 1)U^{*}=V_{\underline{t}}
\]
for every $\underline{t} \in \bbr_{+}^{d}$, where  $\{S_{\underline{t}}\}_{\underline{t} \in \bbr_{+}^{d}}$ is the shift semigroup on $L^2(\bbr_{+}^{d})$ defined by \begin{align}S_{\underline{t}}f(\underline{x}):=\begin{cases}
		f(\underline{x}-\underline{t}), & \textrm{ if  $\underline{x}-\underline{t} \in \bbr_+^d$,} \\
			0 & \textrm{otherwise}.
		\end{cases}
		  \end{align}
\end{theorem}

The above theorem together with an application of Wold decomposition gives a complete list of irreducible doubly commuting isometric representations of $\bbr_{+}^{d}$ on Hilbert spaces. For example, for $d=2$, the set of irreducible doubly commuting isometric representations can be identified with $\bbr^{2} \sqcup \bbr \sqcup \bbr \sqcup \{\star\}$, where $\star$ corresponds to the shift semigroup on $L^2(\bbr_{+}^{2})$. The above forms the content of Section 2. 

In the remaining sections, we focus our attention on the study of the Fell topology on the set of irreducible doubly commuting isometric representations of $\bbr_{+}^{d}$ on Hilbert spaces when $d$ is finite as well as when $d$ is infinite.  To understand how complex the representation theory of a group $G$ (or a $C^{*}$-algebra $A$), it is classical to consider the dual $\widehat{G}$ (or $\widehat{A}$) which is the set of equivalence classes of irreducible representations of $G$ with the Fell topology. Whether the topology on $\widehat{G}$ is good or pathological indicates whether the representation theory of $G$ is `tractable' (type I) or `wild'. We consider a similar notion for isometric representations of $\bbr_{+}^{d}$ and $\bbr_{+}^{\infty}$. We show that if restrict attention to the doubly commuting case, then the topology is $T_0$ for $\bbr_{+}^{d}$ when $d<\infty$. However,   the topology is not $T_0$ for $\bbr_{+}^{\infty}$. In fact, by considering the analogue of shift semigroups on subspaces of $L^2(\bbr^\bbn,\bigotimes_{n=1}^{\infty}\gamma)$ (where $\gamma$ is the standard  Gaussian measure),  we construct a continuum of irreducible, strongly pure doubly commuting isometric representations which are mutually inequivalent but whose equivalence classes have the same closure in the Fell topology.  This says that the representation theory of $\bbr_{+}^{\infty}$ is `wild'. 

We must mention here that in the finite case, i.e. the case of $\bbr_{+}^{d}$ with $d<\infty$, the Fell topology can be analysed via $C^{*}$-algebraic methods by considering the associated Wiener-Hopf algebra (see Thm. 7.8 of \cite{Sundar_Ore}). This is quite well known for $d=1$. However, up to the authors' knowledge, the details have not been explicitly spelt out for $d>1$.  We chose to work directly with the isometric representations and avoid the language of $C^{*}$-algebras. One reason is that $C^{*}$-algebraic techniques do not adapt well in the non-locally compact space, and for semigroups like $\bbr_{+}^{\infty}$ that are not locally  compact. 

As yet another pathology, we show by an example that doubly commuting isometric representations of $\bbr_{+}^{\infty}$ need not admit Wold decomposition. 

The organisation of this paper is described below.

After this introduction, in Section 2, we record our new proof of Cooper's theorem, and prove Thm. \ref{main_intro}.  In Section 3, we consider the Fell topology on the set of equivalence classes of irreducible isometric representations of a semigroup $S$. We follow \cite{Ernest} and make use of railway representations to define the Fell topology. We show that if we restrict  to the doubly commuting case and when $S=\bbr_{+}^{d}$ with $d$ finite, then the Fell topology is $T_0$. In Section 4, we indicate the pathologies that appear when $d=\infty$.  We prove that the Fell topology on the set of equivalence classes of irreducible doubly commuting isometric representations of $\bbr_{+}^{\infty}$ is not $T_0$. We also give an example of a doubly commuting isometric representation of $\bbr_{+}^{\infty}$ which does not have a  Wold decomposition.

\textbf{Notation:} 
\begin{itemize}
    \item For a measure space  $(X,\clb,\mu)$, and for  $\phi \in L^{\infty}(X,\clb,\mu)$, $M_\phi$ stands for the multiplication operator on $L^2(X,\mu)$ defined by 
    \[
    M_\phi(f)=\phi f.
    \]
    \item The Hilbert spaces  that we consider are assumed to be separable.
    \item  We denote $\bbn=\{1,2,\cdots\}$, $\bbn_0=\bbn \cup \{0\}$,  and $\bbr_{+}=[0,\infty)$.
        \item We set $\bbr^\infty=\{(x_n) \in \bbr^\bbn: \textrm{$x_n=0$ eventually}\}$ and   endow $\bbr^\infty$ with the inductive limit topology. For $d \in \{1,2,\cdots\} \cup \{\infty\}$, $\{e_i\}_{i=1}^{d}$ denotes the standard basis for $\bbr^d$.
  \end{itemize}

 \section{Cooper's theorem}

Let $\cla$ be a $C^{*}$-algebra, and let $E$ be a Hilbert $\cla$-module. For  $d \in \mathbb{N}$, let   $I_d=\{1,2, , \cdots , d\}$.  Let  
$T=\{T_{ \underline{t} }\}_{ \underline{t} 
\in \bbr^d_+ }$ be a strictly continuous semigroup of adjointable operators (also called a strictly continuous representation of $\bbr_{+}^{d}$) on $E.$ For  each $i \in I_d$ and $t \in \bbr _+,$   denote $T_{t\underline{e}_i}$ by $T_t^{(i)},$  where $e_i$ is the $i$-th standard basis vector in $\bbr^d$. A strictly continuous representation  $T$ is said to be \emph{doubly commuting} if, for all  distinct  $i,j \in I_d,$  $$T^{(i)}_{s}T_{t} ^{(j)*}=T_{t} ^{(j)*}T^{(i)}_{s} \ \mbox{for all }\, s,t \in \bbr_+  .    $$
The strictly continuous semigroup $T$ is said to be   strongly pure if $\displaystyle \lim _{t \rightarrow \infty}T_{t} ^{(i)*} = 0 $ in the strong operator topology (SOT) for all $i \in I_d.$  The representation $T$ is  said to be isometric if $T_{\underline{t}}$ is an isometry for all $\underline{t} \in \bbr_+^d.$
Since  we consider only strictly (strongly)  continuous isometric representations on Hilbert modules (Hilbert spaces), we often drop the adjective `strictly (strongly) continuous' and merely call them isometric representations.

For  $d \in \mathbb{N},$  define an isometric representation $S:=\{S_{\underline{t}}\}_{\underline{t} \in \bbr_+^d}$ on the Hilbert  space  $L^2(\bbr_+^d)$ by 
\begin{align}S_{\underline{t}}f(\underline{x}):=\begin{cases}
		f(\underline{x}-\underline{t}), & \textrm{ if  $\underline{x}-\underline{t} \in \bbr_+^d$,} \\
			0, & \textrm{otherwise}.
		\end{cases}
		  \end{align}
    It is straightforward to verify that $S$ is doubly commuting and for each $i \in I_d$, $\{S_{t}^{(i)}\}_{t \geq 0}$ is pure. We call $S=\{S_{\ut}\}_{\ut \in \bbr_+^d}$ the shift semigroup of $\bbr_{+}^{d}$ on $L^2(\bbr_+^d)$. Let $$\mathbb{H}^d_+:= \{ \underline{z} \in \bbc^d \, : \, Re(z_i)>0 \textrm{~for all $i \in I_d$}\}.$$ For $\underline{z} \in \mathbb{ H}^d_+,$ define the unit vector  $f_{\underline{z}}$  in $L^2(\bbr_+^d)$ by 
 \begin{align}\label{SO}
 f_{\underline{z}}(\underline{x})=  \sqrt{2^d \prod_{i \in I_d} Re(z_i)} e^{-\langle \underline{z},\underline{x}\rangle }.
 \end{align}Observe that, for  $\underline{t} \in \bbr_{+}^d,$
 \begin{align}\label{S5}
 S_{\underline{t}}^*f_{\underline{z}}=e^{-\langle \underline{z},\underline{t} \rangle } f_{\underline{z}}.
 \end{align} 
 The vector $f_{\underline{z}}$ and the above equation will play a crucial role in what follows.

 Let $\cla$ be a separable $C^{*}$-algebra, and consider a Hilbert  $\mathcal{A}$-module $K$ which we assume is countably generated.  Let $L^2(\bbr_+^d)  \otimes K$ denote the   external tensor product  of  the Hilbert modules $L^2(\bbr_+^d)$ (viewed as a Hilbert $\mathbb{C}$-module) and $K$.  The representation $S \otimes 1_{K} = \{S_{\underline{t}} \otimes 1\}_{\underline{t} \in \bbr_+^d}$   is a  doubly commuting isometric representation of $\bbr_{+}^{d}$ on the Hilbert $\mathcal{A}$-module $L^2(\bbr_+^d) \otimes K$. Moreover, it is strongly pure.   In this section,  we  will demonstrate  that every  doubly commuting isometric representation  $V=\{V_{ \underline{t}}\}_{ \underline{t} \in \bbr^d_+ }$ on a Hilbert $\mathcal{A}$-module $E$ which is strongly pure   is unitarily equivalent to the representation $S \otimes 1_{K}=\{S_{\underline{t}} \otimes 1_K\}_{\underline{t} \in \bbr_+^d}$ for some Hilbert $\cla$-module $K$.

 The proof of the following proposition is included for completeness as it is probably known. 
 \begin{prop} \label{PP1}
    Let $T=\{T_{\underline{t}}\}_{\underline{t} \in \mathbb{R}_+^{d }} $ be a semigroup of bounded linear operators on a separable Banach space $E$. Suppose that $T$ is weakly continuous, i.e. for every $\phi \in E^{*}$, the map 
    \[
    \bbr_{+}^{d} \ni \ut \to \phi(T_{\ut}\xi) \in \bbc\]
    is continuous for every $\xi \in E$. Assume that $T$ is $\mathbb{N}^d$-periodic,  i.e. $T_{\underline{t}+\underline{n}}=T_{\underline{t}}$ for all $\underline{n}\in \mathbb{N}^{d }.$  For $\underline{n}
    \in \mathbb{Z}^d,$ let  \[\ E_{\underline{n}}: = \{\xi \in E :  T_{\underline{t}} \xi = e^{2\pi i \left <\underline{n}, \underline{t}\right > }\xi  \,\, \mbox{for all} \,\, \underline{t} \in \mathbb{R}_+^{d}\}.\] Then, 
    \[E= span   \overline { \{E_{\underline{n}} \, : \, \underline{n} \in \mathbb{Z}^d  \}}.\]
    In particular, $E_{\underline{n}}$ is non-zero for some $\underline{n} \in \bbz^d$. 
    \end{prop}
 \begin{proof}
For $\underline{n} \in \mathbb{Z}^d$, define a linear operator  $P_{\underline{n}}$ on $E$ by
\[P_{\underline{n}}:= \int_{[0, 1]^d }e^{-2\pi i \left < \underline{n}, \underline{t}\right > } T_{\underline{t}}   \,d\underline{t}.\]
 Since $T$ is periodic,  we have for   $\xi \in E $ and $\underline{s} \in \mathbb{R}_+^{d },$ $$T_{\underline{s} }P_{\underline{n}}\xi = \int_{[0, 1]^d }e^{-2\pi i \left < \underline{n},\underline{t}\right > } T_{\underline{t}+\underline{s}}\xi   \,d\underline{t}=e^{2\pi i \left <\underline{n},\underline{s}\right > } \int_{[0, 1]^d }e^{-2\pi i \left < \underline{n},\underline{t}\right > } T_{\underline{t}} \xi  \,d\underline{t}=e^{2\pi i \left <\underline{n},\underline{s}\right > }P_{\underline{n}}\xi.$$ 
Therefore, $Ran(P_{\underline{n}}) \subset E_{\underline{n}}$.  It is clear that if $\xi \in E_{\underline{n}}$, then $P_{\underline{n}}\xi=\xi$. Hence,    $E_{\underline{n}}= Ran (P_{\underline{n}})$ for all $\underline{n}
 \in \mathbb{Z}^d$. 
 Next, we prove that $span \{ E_{\underline{n}} \, : \, \underline{n} \in \mathbb{ Z}^d\}$
 is dense in $E.$ If not, then  there exists a non-zero linear functional $\phi \in E^*$  that vanishes on $span \{ E_{\underline{n}} \, : \, \underline{n} \in \mathbb{ Z}^d\}.$ Then, for every $\underline{n} \in \mathbb{Z}^d$ and $\xi \in E,$  $$\int_{[0,1]^d} e^{-2\pi i \left < \underline{n},\underline{t}\right > }\phi(T_{\underline{t}}\xi)\, d\underline{t}=\phi(P_{\underline{n}} \xi)=0. 
 $$
 Since $T$ is weakly continuous, and since the Fourier coefficients of  the continuous function \[ [0,1]^{d} \ni \underline{s} \mapsto  \phi(T_{\underline{s}} \xi)  \in \mathbb{C}\]   vanish for every  $\xi \in E,$  $ \phi(T_{\underline{s}} \xi)=0$ for every $\underline{s} \in [0,1 ]^{d}$ and $\xi \in E.$ In particular, for $\underline{s}=0,$ we get $\phi(\xi)=0$ for all $\xi \in E.$ This contradicts  the fact that $\phi$ is nonzero. Therefore, $E=\overline{span \{ E_{\underline{n}} \, : \, \underline{n} \in \mathbb{ Z}^d\}}$.
 \end{proof}

 Fix a separable $C^{*}$-algebra $\cla$, and a countably generated Hilbert $\cla$-module $E$. Suppose $V=\{ V_{ \underline{t}}\}_{ \underline{t} \in \bbr^d_+ }$ is a  strictly continuous, doubly commuting isometric representation  on the Hilbert $\mathcal{A}$-module $E$. We further assume that $V$ is strongly pure, i.e. for every $i$, $V_{te_i}V_{te_i}^{*} \to 0$ (in SOT) as $t \to \infty$. The isometric representation $V$ is fixed until further mention.

  \begin{lemma}
  \label{chaning the order}
  For $\underline{s},\underline{t} \in \bbr_+^d,$ there exist $\underline{p},\underline{q} \in \bbr_+^d$ such that $\underline{t}-\underline{s}=\underline{p}-\underline{q}$, and  
   \begin{align}\label{DOC} V_{\underline{s}}^*V_{\underline{t}}=V_{\underline{p}}V_{\underline{q}}^* \end{align}
   \end{lemma}
   \textit{Proof.} Let $I=\{i \in I_d:t_i \geq s_i\}$. Set $\underline{p}:=\sum_{i \in I}(t_i-s_i)e_i$ and $\underline{q}:=\sum_{i \in I^c}(s_i-t_i)e_i$.  Thanks to the fact that $V$ is doubly commuting, $V_{\underline{s}}^{*}V_{\ut}=V_{\underline{p}}V_{\underline{q}}^{*}$. \hfill $\Box$.

  For $\underline{z}  \in \mathbb{H}^d_+, $ define the submodules $K_{\underline{z}}$ and $L_{\underline{z}}$ as  $$K_{\underline{z}}:=\{  \xi \in E \, : \, V_{\underline{t}}^*\xi= e^{-\langle \underline{z},\underline{t}\rangle }\xi \, \mbox { for all } \, \underline{t} \in \bbr_+^d\},$$ and
  $$L_{\underline{z}}:= span \overline{ \{ V_{\underline{t}}\xi \, : \, \xi \in K_{\underline{z}}, \underline{t} \in \bbr_+^d\}}.$$
  Note that $K_{\underline{z}}$ is invariant for the representation  $V^*=\{ V^*_{ \underline{t}}\}_{ \underline{t} \in \bbr^d_+ },$ i.e. $V^*_{ \underline{t}}K_{\underline{z}} \subset K_{\underline{z}} $ for all  $\underline{t} \in \bbr_+^d$.
   By Lemma \ref{chaning the order}, $L_{\underline{z}}$ is reducing for $V$, i.e. $L_{\underline{z}}$ is invariant under both the representations  $V$ and $V^*.$   To prove the main theorem (Theorem \ref{thm1}), we  show that for all $\underline{z} \in \mathbb{H}_+^{d},$  $ K_{\underline{z}} $ is nonzero    and $L_{\underline{z}}$ is independent of $\underline{z}.$

   \begin{prop} \label{PP2}
       If $K_{{\underline{z}}_0} \neq 0$ for some ${\underline{z}}_0 \in \mathbb{H}_+^{d}$, then $K_{\underline{z}} \neq 0$ for all $\underline{z} \in \mathbb{H}^d_+.$
   \end{prop}
 \begin{proof}

Assume that $K_{\underline{z}_0}$ is nonzero for some ${\underline{z}}_0 \in \mathbb{H}^d_+.$  Let $\xi$ be a nonzero vector in $K_{\underline{z}_0}$. Define the   submodule $W$  by   $$W:= \overline {span \{ V_{\underline{t}} \xi a \,  : \,a \in \mathcal{A}, \underline{t} \in \bbr_+^d \}}.$$ 

Using Eq. \ref{DOC} and the fact that $\xi \in K_{{\underline{z}}_0},$ we see that $W$ is      reducing for $V.$ Let $L_{\xi}$ be the submodule of $E$ generated by $\xi,$  i.e. $L_{\xi}=\overline{span\{\xi a \, : \, a \in \mathcal{A}\}}.$ We claim that $V|_{W}$ is unitarily equivalent to the  shift representation  $S \otimes 1_{L_{\xi}}=\{ S_{\underline{t}} \otimes 1\} _{\underline{t} \in \bbr_+^d}$ on $L^2(\bbr_+^d) \otimes L_{\xi}.$ The proof will be over once we show this. For,  we can then restrict $V$ to the submodule $W$ and assume  without loss of generality that $V$ is the shift representation  $S \otimes 1=\{ S_{\underline{t}} \otimes 1_{L_{\xi}}\} _{\underline{t} \in \bbr_+^d}$ on $L^2(\bbr_+^d) \otimes L_{\xi}.$ However,  it follows from Eq. \ref{S5} that  the subspace $K_{\underline{z}} $ for the shift representation  $S \otimes 1=\{ S_{\underline{t}} \otimes 1\} _{\underline{t} \in \bbr_+^d}$ is nonzero  for every $\underline{z} \in \mathbb{H}_+^d$.

Let $\underline{s}, \underline{t} \in \bbr_+^d$. Thanks to Lemma \ref{chaning the order}, there exists $\underline{p},\underline{q} \in \bbr_{+}^{d}$ such that $\underline{t}-\underline{s}=\underline{p}-\underline{q}$ and $V_{\underline{s}}^{*}V_{\ut}=V_{\underline{p}}V_{\underline{q}}^{*}$. Then, for $a,b \in A$, 
\begin{align*}
 \langle  V_{\underline{s}} \xi a, V_{\underline{t}} \xi b \rangle &= \langle \xi a, V_{\underline{s}}^{*}V_{\ut}  \xi b \rangle \\
&=\langle \xi a|V_{\underline{p}}V_{\underline{q}}^{*}(\xi b)\rangle \\
&= \langle V_{\underline{p}}^{*}(\xi a)|V_{\underline{q}}^{*}(\xi b)\rangle \\
&=e^{-(\langle \underline{z}_0,\underline{p}\rangle+\overline{\langle \underline{z}_0,\underline{q}\rangle})} \langle \xi a|\xi b\rangle. 
\end{align*}
By a similar computation, using Eq. \ref{S5}, we get 
\[
\langle S_{\underline{s}}f_{\underline{z_0}} \otimes \xi a|S_{\underline{t}}f_{\underline{z_0}}\otimes \xi b \rangle =e^{-(\langle \underline{z}_0,\underline{p}\rangle+\overline{\langle \underline{z}_0,\underline{q}\rangle})} \langle \xi a|\xi b\rangle. 
\]
Hence, for $a,b \in \cla$, $\underline{s},\underline{t} \in \bbr_{+}^{d}$,
\begin{equation}
    \label{existence of unitary}
     \langle  V_{\underline{s}} \xi a, V_{\underline{t}} \xi b \rangle=\langle S_{\underline{s}}f_{\underline{z_0}} \otimes \xi a|S_{\underline{t}}f_{\underline{z_0}}\otimes \xi b \rangle.
\end{equation}

Note that $\{ V_{\underline{t}} \xi a \,  : \,a \in \mathcal{A},  \underline{t} \in \bbr_+^d \}$ and $\{S_{\underline{t}}f_{{\underline{z}}_0} \otimes \xi a\, : \, a \in \mathcal{A}, \underline{t} \in \bbr_+^d\}$ are total sets in $W$ and $L^2(\bbr_+^d) \otimes L_{\xi}$, respectively. Thus, thanks to Eq. \ref{existence of unitary},   there exists a unitary module map $U: W \rightarrow L^2(\bbr_+^d) \otimes L_{\xi}$ such that \[UV_{\underline{s}} \xi a =S_{\underline{s}}f_{\underline{z}_0} \otimes \xi a \]  for all $a \in \mathcal{A}$ and $\underline{s} \in \bbr_+^d.$     Moreover, for $\underline{t} \in \bbr_+^d$, 
$$ UV_{\underline{t}} (V_{\underline{s}} \xi a )=UV_{\underline{t}+\underline{s}} \xi a =S_{\underline{t}+\underline{s}}f_{{\underline{z}}_0} \otimes \xi a =(S_{\underline{t}} \otimes 1)U(V_{\underline{s}} \xi a),$$
 for all $ a\in \mathcal{A} $ and $\underline{s} \in \bbr_+^d$.
Since $\{ V_{\underline{t}} \xi a \,  : \,a \in \mathcal{A}, \underline{t} \in \bbr_+^d \}$ is a total set in $W$,  $UV_{\underline{t}}=(S_{\underline{t}} \otimes 1)U$ for all $\underline{t} \in \bbr_+^d$, i.e. $U$ intertwines the representations $V|_{W}$ and $S \otimes I_{L_{\xi}}.$
 \end{proof}

\begin{remark}\label{DV}
    
Let $\{V_1, V_2, \cdots , V_d\} $ be a  family of isometries on  $E$ such that each $V_i$ is pure for every $i$, i.e. $V_i^{*n} \rightarrow 0$ in the strong operator topology $(SOT)$. Suppose that $\{V_1,V_2,\cdots,V_d\}$ is doubly commuting, i.e. $V_iV_j=V_jV_i$ for all $i,j$ and $V_{i}^{*}V_j=V_jV_{i}^{*}$ for $i \neq j$. 
Note that
 $$\sum_{n \in \mathbb{Z}_+}(V_i^{n}V_i^{*n}-V_{i}^{n+1}V_{i}^{*n+1})=\lim_{n \rightarrow \infty } (I-V_i^{n}V_i^{*n})=I.$$  From this, we can decompose $E$ as
    \begin{align}\label{ONE}
    E= \bigoplus _{n \in \mathbb{Z}_+}V_i^{n}  (ker V_i^*).
    \end{align}
    for each $i \in I_d$.

    For each  $ i \in I_d,$  by  the definition of doubly commuting  isometries,  the submodule  $ker V_i^{*}$ is $V_j$-reducing   for  all $j \in I_d$ with $i \neq j$, and  the restriction of $V_j$ to $ker V_i^{*}$ is a pure isometry.  Thus,  by  induction and by applying  Eq. \ref{ONE},  we get  $$E= \bigoplus _{\underline{n} \in \mathbb{Z}^d_+}V_{\underline{n}}  ( \bigcap_{i \in I_d} ker V_i^{*}).$$   Hence, $\bigcap_{i \in I_d} ker V_i^{*}$ is non-zero .
    \end{remark}

\begin{prop}\label{PP3}
For every $\underline{z} \in \mathbb{H}_+^d, K_{\underline{z}} \neq 0.$
\end{prop}
\begin{proof}
    In view of  Prop. \ref{PP2}, it is sufficient to show that $K_{\underline{z}} \neq 0$ for some $\underline{z} \in \mathbb{H}_+^d.$ Let $\underline{z} \in \mathbb{H}_{+}^{d}$. 
    Let
    \begin{align}\label{DSV}
    E_0=\{ \xi \in E \, : \, V^*_{\underline{n}}\xi = e^{-\langle \underline{z}, \underline{n} \rangle} \xi \textrm{~for all~}\underline{n} \in \mathbb{Z}^d_+  \}.\end{align}
    First, we  prove that $E_0$ is non-zero. Since $V$ is strongly pure, for each $i \in I_d, $  $V^{(i)}_1$ is a pure isometry and $\{V^{(1)}_1,V^{(2)}_1, \cdots , V^{(d)}_1\}$ is doubly commuting. By Remark \ref{DV},
    there exists $\xi \in \bigcap_{i \in I_d} ker V_1^{(i)*}$ such that $\xi \neq 0$. It is easily verifiable that $\sum_{\underline{n} \in \mathbb{Z}^d_+} e^{-\langle \underline{\beta, n} \rangle }V_{\underline{n}} \xi \in E_0$. Hence,  $E_0$ is nonzero.
    
     Note that the submodule $E_0$  is  invariant for the representation  $V^*=\{ V^*_{ \underline{t}}\}_{ \underline{t} \in \bbr_+^d}$. For $\ut \in \bbr_{+}^{d}$,  define an  operator $T_{\underline{t}}: E_0 \rightarrow E_0$ by
     \begin{align}\label{DSV1}
     T_{\underline{t}}\xi= e^{\langle \underline{z} , \underline{t} \rangle }V^*_{\underline{t}} \xi,  \,\,\, \xi \in E_0.\end{align}
     
     Then, $T=\{T_{\underline{t}}\}_{\underline{t} \in \bbr_+^d}$ is a strongly continuous semigroup on  the Banach space $E_0.$ Moreover, $T$ is $\mathbb{N}^d$-periodic.   By Prop. \ref{PP1}, there exists a nonzero vector $\xi \in E_0$ and $\underline{n} \in \bbz ^d$ such that \[T_{\underline{t}} \xi= e^{-2\pi i \langle \underline{n},\underline{ t}  \rangle }\xi\] for all $\underline{t} \in \bbr_+^d,$  i.e.  $V^*_{\underline{t}} \xi= e^{- \langle\underline{z}+2\pi i \underline{n}, \underline{t} \rangle }\xi.$ This shows that   $K_{\underline{z} +2\pi i \underline{n}} \neq 0.$
    \end{proof}

\begin{prop}\label{IND}
For  $\underline{z}_1, \underline{z}_2  \in \mathbb{H}_+^d, L_{\underline{z}_1}=L_{\underline{z}_2}.$ 
\end{prop}
\begin{proof}
     Let $\underline{z}_1, \underline{z}_2  \in \mathbb{H}_+^d. $  Since $L_{\underline{z}_i}, i=1,2 ,$ is a reducing submodule for  $V,$  it is sufficient to show that  $ K_{\underline{z}_1} \subset L_{\underline{z}_2} .$ Let $\xi  \in K_{\underline{z}_1}$, and define the submodule $W_{\xi}$  of $E$ by $$W_{\xi}:=span \overline{ \{ V_{\underline{t}} \xi a \,  : \,a \in \mathcal{A}, \underline{t} \in \bbr_+^d \} }.$$
     Then, thanks to Lemma \ref{chaning the order}, $W_{\xi}$ is  reducing for the representation $V$. 
Let $L_{\xi}: =\overline{span\{\xi a  \, : \, a \in \mathcal{A}\}}$. As in Lemma \ref{PP2}, (using Eq. \ref{existence of unitary}), we see that there exists a unitary  $U:W_{\xi} \rightarrow L^2(\bbr_+^d)\otimes L_\xi$ such that \[U(V_{\underline{t}}\xi a )= S_{\underline{t}}f_{\underline{z}_1} \otimes  \xi a\]
for $\underline{t} \in \bbr_{+}^{d}$. Moreover, $U$  intertwines the representations $V|_{W_\xi}$ on $W_{\xi}$ and $S \otimes 1_{L_{\xi}}$ on $L^2(\mathbb{R}_+^d) \otimes L_{\xi},$ i.e.  $ UV_{\underline{t}}\eta=(S_{\underline{t}} \otimes 1)U\eta$ for all $\eta \in W_{\xi}$. 

Since $span \{S_{\underline{t}}f_{\underline{z}_2}:  \underline{t} \in \mathbb{R}_+^d\}$ is dense in $ L^2(\bbr_+^d),$ there exists a sequence $(h_{\underline{t}_n})_{n \in \bbn}$ in $L^2(\bbr_+^d)$ converging to $f_{\underline{z}_1}$ and is of the form 

$$h_{\underline{t}_{n}}=\sum_{m=0}^{N_{\underline{t}_n}}\alpha_{\underline{t}_m}S_{\underline{t}_m}f_{\underline{z}_2}, \hspace{0.6cm}\alpha_{\underline{t}_m} \in \mathbb{C}, \underline{t}_m\in\bbr_+^d\, \mbox{and}\,  N_{\underline{t}_n} \in \mathbb{N}_0.  $$
   Now,
   $$
   \xi a =  U^{*}(f_{\underline{z}_1} \otimes  \xi a)
   =\lim_{ n \rightarrow \infty } \sum_{m=0}^{N_{\underline{t}_n}}\alpha_{\underline{t}_m}U^*(S_{\underline{t}_m}\otimes 1).(f_{\underline{z}_2} \otimes \xi a).$$
   This simplifies to
   $$
\xi a=\lim_{ n \rightarrow \infty } \sum_{m=0}^{N_{\underline{t}_n}}\alpha_{\underline{t}_m}V_{\underline{t}_m}U^*(f_{\underline{z}_2} \otimes \xi a),
   $$
   where the last equality follows from the fact that  $U$ intertwines the representation $V|_{W_\xi}$ and $S \otimes 1_{L_{\xi}}$. Furthermore, since  $$V^*_{\underline{t}}U^*(f_{\underline{z}_2} \otimes \xi a)=U^* (S^*_{\underline{t}}f_{\underline{z}_2} \otimes \xi a) =e^{-\langle \underline{z}_2, \underline{t} \rangle }U^*(f_{\underline{z}_2} \otimes \xi a)$$ for all $\underline{t} \in \bbr_+^d,$ we conclude that   $U^*(f_{\underline{z}_2} \otimes \xi a) \in K_{\underline{z}_2}.$ This implies that $\xi a \in L_{\underline{z}_2} $ for all $a \in \mathcal{A}.$ For an approximate identity $\{e_\lambda\}_{\lambda \in \Lambda}$ of $\cla$, $\xi e_\lambda \to \xi$. Hence, $\xi \in L_{\underline{z}_2}$. Therefore,  $K_{\underline{z}_1} \subset L_{\underline{z}_2}.$  This completes the proof.
    \end{proof}

 We next prove Thm. \ref{main_intro}. In fact, we show that we can take $K=K_{\underline{z}}$. 
 
 \textit{Notation:}  For $\underline{\epsilon} \in \{0,1\}^{d}$, let $|\underline{\epsilon}|=\sum_{i=1}^{d}\epsilon_i$.
\begin{theorem}\label{thm1}
     Keep the foregoing notation. For every $\underline{z} \in \mathbb{H}_+^d,$ there exists a   unitary  operator $U: E \rightarrow L^2(\bbr_+^d)  \otimes K_{\underline{z}} $ such that 
 \begin{enumerate}
     \item  for $\underline{t} \in \bbr_+^d, UV_{\underline{t}}=(S_{\underline{t}} \otimes 1)U,$ and 
      \item for $\xi \in K_{\underline{z}}$, $U(\xi )=f_{\underline{z}} \otimes \xi$. 
 \end{enumerate}
     \end{theorem}
    \begin{proof} 
Let  $\underline{z} \in \mathbb{H}_{+}^{d}$. Let $E_0$ be the submodule  of $E$ 
and  $T=\{T_{\underline{t}}\}_{\underline{t} \in \bbr_+^d}$  the representation   defined as in Eq. \ref{DSV}  and Eq. \ref{DSV1}, respectively.  By  applying Prop. \ref{PP1} to the representation $T$ on $E_0$, we see that   $ span    \{(E_0)_{\underline{n}} \, : \, \underline{n} \in \mathbb{Z}^d  \}$ is dense in $E_0$, where  \[(E_0)_{\underline{n}}: = \{\xi \in E_0 :  T_{\underline{t}} \xi = e^{2\pi i \left <\underline{ n}, \underline{t}\right > }\xi  \,\, \mbox{for all} \, \, \underline{t} \in \mathbb{R}_+^{d}\} .\]
For $ \underline{n} \in \mathbb{Z}^d$ and  $\xi \in (E_0)_{\underline{n}}$, we have  $$V^*_{\underline{t}}\xi= e^{-\langle \underline{z}, \underline{t} \rangle }T_{\underline{t}}\xi=e^{- \langle \underline{z} +2\pi i \underline{n}, \underline{t} \rangle }\xi, \,\,  \,\,   \underline{t}\in \bbr_+^d.$$ This implies that $(E_0)_{\underline{n}} \subset  K_{\underline{z} +2\pi i \underline{n}}.$ Furthermore, By Prop. (\ref{IND})
$$(E_0)_{\underline{n}} \subset  K_{\underline{z} +2\pi i \underline{n}}\subset L_{\underline{z} +2\pi i \underline{n}} =L_{\underline{z}}, \, \, \mbox{for all}\,\,  \underline{z} \in \mathbb{H}^d_+ \, \mbox{and}\, \underline{n}\in \mathbb{Z}^d.$$
Therefore,  $E_0 \subset L_{\underline{z}}$ for every $\underline{z} \in \mathbb{H}_+^d.$

Let  $\xi \in \bigcap_{i \in I_d} ker V_1^{(i)*}$, and   define   $\eta:= \sum_{\underline{n} \in \mathbb{Z}^d_+} e^{-\langle \underline{z}, \underline{n} \rangle }V_{\underline{n}}\xi$.  It is routine to verify that  $ \eta \in E_0$  and   $\displaystyle \xi=\sum_{\underline{\epsilon}\in \{0,1\}^{d}}(-1)^{|\underline{\epsilon}|}e^{-\langle \underline{z}|\underline{\epsilon}\rangle}V_{\underline{\epsilon}}\eta$.  Hence,  
$\xi \in  \overline{span \{V_{\underline{n}}E_0 \, : \, \underline{n} \in \mathbb{Z}_+^d \}}\subset L_{\underline{z}}$ which implies that  $\bigcap_{i \in I_d} ker V_1^{(i)*} \subset L_{\underline{z}}.$ Since the submodule $L_{\underline{z}}$ is reducing for the representation $V,$ by Remark \ref{DV}, $E= \bigoplus _{\underline{n} \in \mathbb{Z}^d_+}V_{\underline{n}}  ( \bigcap_{i \in I_d} ker V_i^{*}) \subset  L_{\underline{z}}.$ Therefore, $E=L_{\underline{z}}$ for every $\underline{z} \in \mathbb{H}_{+}^{d}$.

Fix $\underline{z} \in \mathbb{H}_{+}^{d}$. A  calculation similar to the one used to deduce  Eq. \ref{existence of unitary} shows that   for $\xi, \eta \in K_{\underline{z}}$ and $\underline{s}, \underline{t} \in \mathbb{R}_+^d,$ we  have
$$\langle V_{\underline{s}}\xi, V_{\underline{t}}\eta\rangle =\langle  S_{\underline{s}}f_{\underline{z}} \otimes \xi,  S_{\underline{t}}f_{\underline{z}} \otimes \eta \rangle. $$

Since $E=L_{\underline{z}}$,  the set $\{V_{\underline{s}}\xi \, : \, \xi \in K_{\underline{z}}, \underline{s} \in \mathbb{R}_+^d\}$ is total in $E$. Also,   $\{ S_{\underline{s}}f_{\underline{z}} \otimes \xi \, : \, \xi \in K_{\underline{z}}, \underline{s} \in \mathbb{R}_+^d\}$ is total in $L^2(\mathbb{R}_+^d) \otimes K_{\underline{z}}$.  It follows from  the preceding equation that there exists a  unitary  map $U: E \rightarrow L^2(\mathbb{R}_+^d) \otimes K_{\underline{z}}$ such  that $$U(V_{\underline{s}}\xi)=S_{\underline{s}}f_{\underline{z}} \otimes \xi \,\,\,\,\mbox{for}\, \, \xi \in K_{\underline{z}}, \,\underline{s} \in \mathbb{R}_+^d.$$ Furthermore, for $\underline{t} \in \mathbb{R}_+^d,$  we  have $$UV_{\underline{t}} (V_{\underline{s}} \xi  )=UV_{\underline{t}+\underline{s}} \xi  =S_{\underline{t}+\underline{s}}f_{{\underline{z}}} \otimes \xi  =(S_{\underline{t}} \otimes 1)U(V_{\underline{s}} \xi ),$$
 where  $\xi \in K_{\underline{z}}$ and $ \underline{s} \in \bbr_+^d.$ Since $\{ V_{\underline{s}} \xi  \,  : \, \xi \in K_{\underline{z}}, \underline{s} \in \bbr_+^d \}$ is a total set in  $E, $  it follows that $UV_{\underline{t}}=(S_{\underline{t}} \otimes 1)U$ for all $\underline{t} \in \bbr_+^d$. The proof is over. 
\end{proof}

    In the Hilbert space setting, we can replace strict continuity of $V$ by strong continuity. Notice that in the course of the proof, the only place where we required strict continuity is in Prop. \ref{PP3}, where we needed strict continuity to ensure that  $\{V_{\ut}^*\}_{\ut \in \bbr_{+}^{d}}$ on the Banach space $E_0$ is weakly continuous in order to apply Prop. \ref{PP1}. 
    However, if $\{V_{\ut}\}_{\ut \in \bbr_+^d}$ is a strongly continuous semigroup of isometries on a Hilbert space $\clh$, then $\{V_{\ut}^{*}\}_{\ut \in \bbr_{+}^{d}}$ is weakly continuous. The rest of the arguments do not require any modification. Thus, we get the following theorem which we call Cooper's theorem. The case $d=1$ (due to Cooper (\cite{Cooper}))   is called Cooper's theorem in the literature.

    \begin{theorem}
    \label{Cooper}
    Let $V=\{V_{\ut}\}_{\ut \in \bbr_+^d}$ be a strongly continuous doubly semigroup of isometries on a separable Hilbert space $\clh$. Suppose that $V$ is strongly pure. Then, there exists a Hilbert space $\clk$ and a unitary operator $U:L^2(\bbr_+^d)\otimes \clk \to \clh$ such that 
    \[
    U(S_{\ut} \otimes 1)U^*=V_{\ut}
    \]
    for every $\ut \in \bbr_+^d$. 
    \end{theorem}

\subsection{Wold Decomposition for doubly commuting semigroups of isometries} 
For the rest of this paper, we restrict ourselves to the Hilbert space setting, and we do not consider Hilbert modules. 
In this subsection, we discuss the Wold decomposition of doubly commuting semigroup of isometries which allows us to deduce the structure of irreducible doubly commuting representations. If $S$ is a topological semigroup and $V=\{V_s\}_{s \in S}$ is a strongly continuous semigroup of isometies, we call $V$ \emph{irreducible} if $V$ has no reducing subspaces, or equivalently, $\{V_s,V_{s}^{*}:s \in S\}^{'}=\bbc$. 
The proof of the following lemma, which is well known, is included for completeness. 
 \begin{lemma}
 \label{commutant}
     For every $d \geq 1$, the isometric representation $\{S_{\underline{t}}\}_{t \in \bbr_{+}^{d}}$ on $L^2(\bbr_{+}^{d})$ is irreducible. 
 \end{lemma}
 \textit{Proof.} Let $M:=\{S_{\underline{t}}, S_{\underline{t}}^{*}: \underline{t} \in \bbr_{+}^{d}\}^{'}$. We prove that $M=\bbc$. Clearly, it suffices to prove when $d=1$. 
 Let $d=1$, and let $T \in M$ be given. Then, $T$ commutes with $E_t:=S_tS_t^{*}=M_{1_{[0,t]}}$ for every $t \in \bbr_{+}$. This implies that $T$ commutes with $M_{1_{[a,b]}}$ for every $a,b \in (0,\infty)$. Hence, $T$ commutes with $\{M_{\phi}: \phi \in L^{\infty}(\bbr_+)\}$. Thus, $T$ must be a multiplication operator $M_{\phi}$ for some $\phi \in L^{\infty}(\bbr_+)$. The equation
 \begin{equation}
     \label{covariance}
     S_{t}^{*}M_\phi S_t= S_{t}^{*}TS_t=T=M_\phi
 \end{equation}
 for every $t>0$ implies that for every $t>0$, $\phi(x+t)=\phi(x)$ for almost all $x$. Hence, $\phi$ is constant which implies that $T$ is a scalar. This completes the proof. \hfill $\Box$

Let $V=\{ V_{ \underline{t}}\}_{ \underline{t} \in \bbr^d_+ }$ be a  strongly continuous, doubly commuting isometric representation  on a Hilbert space $\mathcal{H}$.  
For  each $i \in I_d,$ define  the projection $P_i$  on $\mathcal{H}$ by $P_i= \lim_{t \rightarrow \infty} V_{t} ^{(i)}V_{t} ^{(i)*}$ (in  SOT) which exists as $\{V_{t}^{(i)}V_{t}^{(i)*}\}$ is decreasing. Using the fact that $V$ is doubly commuting, we have 
\begin{align}\label{COMM} V_{t} ^{(i)}V_{t} ^{(i)*}V_{s} ^{(j)}V_{s} ^{(j)*}=V_{s} ^{(j)}V_{s} ^{(j)*}V_{t} ^{(i)}V_{t} ^{(i)*}, \, \, s,t \in \bbr_+, \,  i, j \in I_d\,\, \mbox{with}\,\, i \neq j.\end{align}

From this commutation relation, it follows that the projections $P_i$ and $P_j$ commute:
  \begin{align*} P_iP_j=P_jP_i \, \, \, i, j \in I_d.\end{align*}  
For $\alpha  \subset I_d$, define a   projection
     $P_{\alpha}$   on $\mathcal{H}$  by  \begin{align}\label{PROJ}
P_{\alpha}= \prod_{ i \in \alpha}(I-P_{i})\prod_{ j \notin  \alpha}P_{j}.\end{align}
From the commutation relation in Eq. \ref{COMM} and the definition of $P_{\alpha}$ in Eq. \ref{PROJ}, it follows that that $V_{\underline {t}}P_{\alpha}=P_{\alpha}V_{\underline{t}}$ for all $\underline{t} \in \bbr_+^d.$   Thus, $P_{\alpha}\mathcal{H}$ is a reducing subspace for the representation $V$. Let $\alpha_1, \alpha_2 \subset I_d$ be such that $\alpha_1 \neq \alpha_2;$  one can readily verify that $P_{\alpha_1}P_{\alpha_2}=0,$  and  the identity operator  $I$ can be expressed as 
$$I= \prod_{i \in I_d}(P_i \oplus (I-P_i))=\sum _{\alpha \subset I_d}P_{\alpha}=\bigoplus _{\alpha \subset I_d}P_{\alpha}.$$ Consequently, the Hilbert space $\mathcal{H}$ decomposes as
$\mathcal{H}=\bigoplus _{\alpha \subset I_d}\mathcal{H}_{\alpha},$ where $\mathcal{H}_{\alpha}=P_{\alpha}\mathcal{H}.$ 


For $i \in \alpha$ and $\xi \in \mathcal{H}_{\alpha},$ we have $$ \| V_{t} ^{(i)*}P_{\alpha}\xi\|^2=\langle V_{t} ^{(i)}V_{t} ^{(i)*}\xi,\xi\rangle \rightarrow \langle P_{i}P_{\alpha}\xi, \xi \rangle =0$$ as $t \to \infty$.
For $i \notin   \alpha$ and $t \in \bbr_+,$ $V_{t} ^{(i)}V_{t} ^{(i)*}P_{j}=P_{j}V_{t} ^{(i)}V_{t} ^{(i)*}=\lim_{s \rightarrow \infty} V_{s+t} ^{(j)}V_{s+t} ^{(j)*}= P_{j}$ for all $ j \notin \alpha,$ so the semigroup $ \{V^{(i)}_t\}_{t \in \bbr_+}$ acts unitarily on $P_j\clh$. Moreover, since $V^{(i)}_t$ leaves each $P_j$ invariant for $j \notin \alpha$ and due to the commutation relation in Eq. \ref{COMM},  it also preserve $P_{\alpha}\clh$, and thus,  $V^{(i)}_t\mathcal{H}_{\alpha}=\mathcal{H}_{\alpha}.$  Therefore, we conclude the following:

\begin{enumerate}
\item for $i \in \alpha,$ the semigroup $ \{V^{(i)}_t\}_{t \in \bbr_+}$ is strongly pure on $\mathcal{H}_{\alpha},$ and 
\item for $i \notin \alpha, \{V^{(i)}_t\}_{t \in \bbr_+}$ is a  unitary representation on $\mathcal{H}_{\alpha}.$
\end{enumerate}

Combining Thm. \ref{Cooper}, Lemma \ref{commutant} with the above discussion, we get the following theorem.

 \begin{theorem} \label{WDCV}
 $V=\{V_{ \underline{t} }\}_{ \underline{t} 
\in \bbr^d_+ }$ be a strongly continuous doubly commuting   semigroup of isometries   on  a Hilbert space $\mathcal{H}.$ Then,  there exist  unique $2^d$ orthogonal $V$-reducing subspaces $\{\mathcal{H}_{\alpha} \, : \, \alpha \subset I_d\}$ such that    
\begin{enumerate}[(1)]
\item $\mathcal{H}=\bigoplus _{\alpha \subset I_d}\mathcal{H}_{\alpha};$ 
\item   for each $\alpha  \subset I_d,  $ the representation $V_{\alpha}|_{\mathcal{H}_{\alpha}}$ is strongly pure and doubly commuting, and  the representation  $V_{\alpha^c}|_{\mathcal{H}_{\alpha}}$ is unitary, where for $\beta \subset I_d$, $V_{\beta}$  is the representation of $\bbr_{+}^{\beta}$ defined by  $V_{\beta}(\underline{t})=\prod_{i \in \beta}V^{(i)}_{t_i},\,\,\,\,\,\,  \underline{t}=(t_i)_{i \in \beta } \in \bbr_+^{\beta}.$
\end{enumerate}

Suppose that $\clh$ is separable. Then, for each $\alpha\subset I_d,$ there exists a Hilbert space $\mathcal{K}_{\alpha}$ and a strongly continuous unitary representation $W_{\alpha^c}=\{W_{\alpha^c}(\underline{t})\}_{\underline{t} \in \bbr_+^{\alpha^c}}$ on $\mathcal{K}_{\alpha}$ such that the restriction of $V=(V_{\alpha}, V_{\alpha^c})$ to  $\mathcal{H}_{\alpha}$ is unitarily equivanet  to  $(S\otimes 1_{\mathcal{K}_{\alpha}}, 1 \otimes W_{\alpha^c })$  on $L^2(\bbr_+^{\alpha}) \otimes \clk_{\alpha}.$

If $V$ is irreducible, then there exists a unique $\alpha \in I_d$ such that $\clh_\alpha \neq 0$, and in this case, $\clk_\alpha$ is one-dimensional and $W_{\alpha^c}$ is given by a character of $\bbr^{\alpha^c}$.
 \end{theorem}

 \section{Fell topology}
 Let $S$ be a  topological semigroup. We follow \cite{Ernest} to define the Fell topology on the `set of equivalence classes of irreducible isometric representations of S' by making use of railway representations.  Recall that an isometric representation $V$ of $S$ on a separable Hilbert space $\clh$ is said to be irreducible if the commutant $\{V_s,V_{s}^{*}:s \in S\}^{'}=\bbc$. For two isometric representations $V$ and $W$, we write $V \sim W$ if $V$ and $W$ are unitarily equivalent. 
\begin{definition}
An isometric representation $V$ of $S$ on a separable infinite dimensional Hilbert space $\clh$ is called a \textit{railway representation} if $V$ is unitarily equivalent to $\displaystyle \bigoplus_{n=1}^{\infty}V_{0}$ for an irreducible isometric representation $V_{0}$ of $S$. 
\end{definition}
If $V$ is a railway representation and $\displaystyle V=\bigoplus_{n=1}^{\infty}V_0$ for an irreducible representation $V_0$, then, by Schur's lemma, $V_0$ is unique (up to a unitary equivalence). We call $V_0$ the irreducible isometric representation that corresponds to $V$. 
Let $\clh$ be an infinite dimensional separable Hilbert space that is fixed once and for all. The set of (strongly continuous) railway representations of $S$ on $\clh$ is denoted by $\mathscr{R}(S,\clh)$.   By an abuse of notation, we write $\mathscr{R}(S,\clh)$ as $\mathscr{R}(S)$.  The quotient $\frac{\mathscr{R}(S)}{\sim}$ will be denoted by $\mathcal{R}(S)$. As railway representations and irreducible representations are in one-one correspondence, we call $\mathcal{R}(S)$  `the set of equivalence classes of irreducible isometric representations of $S$'. 

We introduce the following topology on $\mathscr{R}(S)$.

 Let $V\in \mathscr{R}(S)$, $\epsilon>0$, $K$ a compact subset  of $S$, and let $F$ be a non-empty finite subset  of $\clh$.  Define
 \[X(V, K, \epsilon, F):=\{ W\in \mathscr{R}(S): \textrm{ $\max_{\xi\in F}\Big(\sup_{x\in K}\big( \{||(V_{x}-W_{x})\xi||$ ,$ ||(V_{x}^{*}-W_{x}^{*})\xi||\}\big)\Big)<\epsilon$}\}.\] 
 The sets of the form $X(V, K, \epsilon,F )$ form a basis for a topology on $\mathscr{R}(S)$. The Fell topology on $\clr(S)$ is defined as the quotient topology.

 Let $[V]\in \clr(S)$, $\epsilon>0$, $K$ a compact subset  of $S$, and let $F \subset \clh$ be  a finite  set of unit vectors.   Define
 \[B([V], K, \epsilon, F):=\Big\{[W]\in \clr(S):\textrm{$\exists \widetilde{W} \sim W ~s.t. \max_{\xi\in F}(\sup_{x\in K}( ||V_{x}-\widetilde{W}_{x})\xi||, ||V_x^*\xi-\widetilde{W}_x^*\xi||)<\epsilon$}.\Big\}\] 
 Notice that the sets of the form $B([V], K, \epsilon, F)$ form a basis for  the Fell topology on $\clr(S)$.

For $d \in \bbn \cup \{\infty\}$, 
Let \[\clr_{0}(\dplus):=\{[V] \in \clr(\dplus): \textrm{$V$ is doubly commuting}\}.\]  

In this section, we assume that $d$ is finite.  We study $\clr_{0}(\dplus)$ with the subspace topology inherited from $\clr(\dplus)$. We call the subspace topology on $\clr_0(\dplus)$  the Fell topology on the set of equivalence classes of irreducible, doubly commuting representations of $\bbr_{+}^{d}$. We prove that $\mathcal{R}_0(\dplus)$ is $T_0$. 


For a  subset $A\subset I_d$, let $Y_{A}$ denote the set of all maps from $ A^c:=I_d \setminus A$ into $\bbr_{+}$. 
For $A \subset I_d$ and  $\lambda\in Y_A$,  define an isometric representation $V^{(A,\lambda)}$ of $\dplus$ on $L^{2}(\dplus)$ by
\[V^{(A,\lambda)}_{\underline{t}}f(\underline{x}):=\begin{cases} e^{i \sum_{j\in A^c}t_{j}\lambda_{j}}f(\underline{x}-\sum_{k\in A} t_{k}e_{k}), & \textrm{ if $x-\sum_{k\in A}t_{k}e_{k}\in \dplus$},\\
0, & \textrm{ otherwise}\end{cases}\] for $\underline{t}=(t_1, t_2,\cdots , t_d)\in\dplus$. The usual convention of interpreting an empty sum as $0$ applies to the above formula when either $A$ or $A^c$ is empty. 

For $A \subset I_d$ and  $\lambda\in Y_A$, set $m:=|A|$, and   define an isometric representation $V^{(A,\lambda,m)}$ of $\dplus$ on $L^{2}(\bbr_{+}^{m})$ by
\[V^{(A,\lambda,m)}_{\underline{t}}f(\underline{x}):=\begin{cases} e^{i \sum_{j\in A^c}t_{j}\lambda_{j}}f(\underline{x}-\sum_{k\in A} t_{k}e_{k}), & \textrm{ if $x-\sum_{k\in A}t_{k}e_{k}\in \bbr_{+}^{m}$},\\
0, & \textrm{ otherwise}\end{cases}\] for $\underline{t}=(t_1, t_2,\cdots , t_d)\in\dplus$. Note that $V^{(A,\lambda,m)}\otimes 1_{L^{2}(\bbr_{+}^{d-m})}$ is unitary equivalent to $V^{(A,\lambda)}$. 

Let $\clk$ be an infinite dimensional Hilbert space, and we take $L^2(\bbr_{+}^{d}) \otimes \clk$ to be the Hilbert space $\clh$ that is fixed so far. 

\begin{theorem}
   With the foregoing notation, the map 
    \[\Psi:\coprod_{A \subset \{1,2\cdots ,d\}}Y_A \to [V^{(A,\lambda)}\otimes 1_{\clk}] \in \clr_{0}(\dplus)\] is a bijection.
    Also, the Fell topology on $\clr_0(\dplus)$ is $T_0$.\end{theorem}
\begin{proof}
    Let $V$ be a doubly commuting irreducible representation of $\dplus$ on a Hilbert space $K$. Thanks to Thm. \ref{WDCV}, Thm. \ref{Cooper} and the fact that an irreducible unitary representation of $\bbr^n$ is  one dimensional and is given by a character of $\bbr^n$, it follows that there exists a unique $A \subset I_d$ and a unique $\lambda \in Y_A$ such that, up to a unitary equivalence, $K=L^{2}(\bbr_{+}^{m})$ (where $m=|A|$) and $V_{\underline{t}}=V^{(A,\lambda,m)}$ for $\underline{t} \in \bbr_{+}^{d}$.
Since $V^{(A,\lambda,m)}\otimes 1_{L^2(\bbr_+^{d-m})}$ is unitarily equivalent to $V^{(A,\lambda)}$, it follows that  a doubly commuting railway representation of $\dplus$ is of the form $V^{(A,\lambda)}\otimes 1_{\clk}$ for a unique $A \subset I_d$ and a unique $\lambda\in Y_{A}$. Also, for $A \subset I_d$ and $\lambda \in Y_A$, thanks to Lemma \ref{commutant}, $V^{(A,\lambda,|A|)}$ is irreducible. Hence,    $V^{(A,\lambda)}\otimes 1_{\clk}$ is a railway representation for every $A \subset I_d$ and $\lambda \in Y_A$.  Hence, $\Psi$ is a bijection. 

Next, we show that the Fell topology on $\clr_{0}(\dplus)$ is $T_0$. Let $A, B \subset I_d$, let $\lambda \in Y_A$ and $\mu \in Y_B$ be given. Denote $V^{(A,\lambda)}\otimes \clk$ by $V$ and $V^{(B,\mu)}\otimes 1_\clk$ by $W$.

Suppose $k \in A \cap B^{c}$. Let $\xi$ be a unit vector in $L^{2}(\bbr_{+}^{d}) \otimes \clk$. We claim that there exists $t>0$ such that $[W] \notin B([V],\{te_k\},\frac{1}{2},\{\xi\})$. Otherwise, for every $t>0$, there exists a unitary operator $U_t \in B(L^2(\bbr_+^d)\otimes \clk)$ such that \[
||V_{te_k}^*\xi-U_{t}W_{te_k}^{*}U_{t}^{*}\xi||<\frac{1}{2}.
\]
However, $W_{te_k}=e^{i\mu_k t}$. Hence, for every $t>0$,
\[
1=||e^{-i\mu_kt}\xi||=||U_tW_{te_k}^{*}U_t^*\xi|| \leq ||U_tW_{te_k}^{*}U_t^{*}\xi-V_{te_k}^{*}\xi||+||V_{te_k}^{*}\xi||<\frac{1}{2}+||V_{te_k}^{*}\xi||.
\]
Notice that since $k \in A$, $\{V_{se_k}\}_{s \geq 0}$ is pure. Thus, letting $t \to \infty$ in the above inequality, we get $1 \leq \frac{1}{2}$ which is a contradiction. 

Thus, if $A \cap B^{c} \neq \emptyset$, then we can find an open set that contains $[V]$ but does not contain $[W]$. 
For the same reason, if $A^{c} \cap B \neq \emptyset$, we can find an open set that contains $[W]$ but  does not contain $[W]$. 

Now, consider the case $A=B$. Suppose there exists $k>0$ such that $\lambda_k \neq \mu_k$. Choose $t>0$ such that \[||e^{i\lambda_k t}-e^{i\mu_k t}|| \geq \frac{1}{2}.\] Let $\xi \in L^2(\bbr_{+}^{d})\otimes \clk$ be a unit vector. Then,  $[W] \notin B([V],\{te_k\},\frac{1}{2},\{\xi\})$. Hence, if $\lambda \neq \mu$, then there exists an open set that contains $[V]$ but that does not contain $[W]$. 

 We can now conclude  that given two distinct points $[V],[W] \in \mathcal{R}_0(\bbr_{+}^{d})$,  there exists an open set which contains one of them but not the other. Hence, $\mathcal{R}_0(\bbr_{+}^{d})$ is not $T_0$. 
\end{proof}

The following proposition says a bit more regarding the Fell topology. We can compute the closure of each point explicitly. 
\begin{prop}
\label{closure}
    For a subset $A \subset I_d$ and $\lambda \in Y_A$, we have 
    \[\overline{\{[V^{(A,\lambda)} \otimes 1_{\clk}]\}}=\{[V^{(B,\mu)}\otimes 1_{\clk}]: B \subset A, \mu \in Y_{B} \textrm{~such that $\mu|_{A^c}=\lambda$}\}.\]
\end{prop}
We can also see using the above proposition that the Fell topology on $\mathcal{R}_0(\bbr_{+})$ is not $T_0$. For, a space being $T_0$ is the same as saying that  distinct points have distinct closures. We do not prove the above proposition but only compute the closure of the singleton $\{[V^{(A,\lambda)}\otimes 1_{\clk}]\}$ when $d=1$ and $A=\{1\}$. This forms the content of the following proposition.

\begin{prop}    
\label{shift_dense} Let $\{S_t\}_{t \geq 0}$ be the shift semigroup on $L^2(\bbr_+)$. 
For $t \geq 0$, let $\widetilde{S}_t=S_t \otimes 1_{\clk}$.  Let $\lambda\in\bbr$, $\epsilon>0$ and $ a>0 $ be given. Suppose $\mathcal{F}$ is a finite set  of unit vectors in $ L^{2}(\bbr_{+})\otimes\clk$. Then, there exists a unitary $W$ on $L^{2}(\bbr_{+})\otimes\clk$ such that for $\xi \in \mathcal{F}$ and $t \in [0,a]$, 
\[\max\Big\{||(e^{i\lambda t}1_{L^{2}(\bbr_{+})\otimes\clk}-W\widetilde{S}_{t}W^{*})\xi||,||(e^{-i\lambda t}1_{L^{2}(\bbr_{+})\otimes\clk}-W\widetilde{S}_{t}^{*}W^{*})\xi||\Big\}<\epsilon.\] 
Hence, $[\widetilde{S}]$ is dense in $\mathcal{R}_0(\bbr_+)$. 
\end{prop}
\begin{proof}  
Choose $\delta>0$ such that \[\sup_{t\in[0,a]}|1-e^{-\delta t}|<\frac{\epsilon^{2}}{2}.\]
Let \[g(x):=\sqrt{2\delta}e^{-(\delta+i\lambda)x}\] for $x\in\bbr_{+}$. Note that $g$ is a unit vector in $L^{2}(\bbr_{+})$, and $S_{t}^{*}g=e^{-(\delta+i\lambda t)}g$. Hence, for $t\in[0.a]$, 
\begin{align*}
||e^{i\lambda t}g-S_{t}g||^{2}=& || e^{i\lambda t}||^{2}+||g||^{2}-2 Re(e^{i\lambda t}\int_{0}^{\infty}S_{t}^{*} g(x) \overline{g(x)} dx)\\
=& 2(1-Re(e^{i\lambda t} e^{-i\lambda t}e^{-\delta t}||g||^{2})\\
=& 2(1-e^{-\delta t})< \epsilon^{2}
\end{align*}
It may similarly be observed that for $t \in [0,a]$,
\[||e^{-i\lambda t}g-S^{*}_{t}g||^{2}<\epsilon^{2}.\]

Suppose $\clf=\{\xi_1, \xi_2 , \cdots , \xi_{r}\}$. We can assume without loss of generality that $\clf$ is orthonormal.  
 Let $U_{1}$ be a unitary  from $L^{2}(\bbr_{+})\otimes \clk$ onto $\displaystyle \oplus_{n=1}^{\infty}L^{2}(\bbr_{+})$ such that $U_1\xi_j$ lies in the $j$-th summand. We identify $L^{2}(\bbr_{+})\otimes \clk$ with $$\cll:=\oplus_{n=1}^{\infty}L^{2}(\bbr_{+})$$ under this unitary, and $\xi_{j}$ with $g_j:=U_{1}\xi_{j}$ for $j=1,2,\cdots ,r$. Let $\{S^{(n)}_{t}\}_{t\geq 0}$ denote the shift semigroup on the $n$-th summand of $\cll$, and  let
 \[S^{'}_{t}:=\bigoplus_{n=1}^{\infty} S^{(n)}_{t}.\] for $t\in\rp$.

For $j=1,2,\cdots ,r$, let $W_{j}:L^{2}(\bbr_{+})\to L^{2}(\bbr_{+})$ be a unitary such that $W_{j}g_{j}=g$. Let $W:=\bigoplus_{n=1}^{\infty}W_j$, where $W_{j}:=1_{L^{2}(\bbr_{+})}$ for $j\geq r+1$. Then, \[\max_{j\in\{1,2,\cdots ,r\}}(\sup_{t\in[0,a]}(||(W^{*}S^{'}_{t}W-e^{i\lambda t}1_{\cll})g_{j}||))=
\sup_{t\in K}(||(S_{t}-e^{i\lambda t}1_{L^{2}(\bbr_{+})})g||)<\epsilon.\] Similarly,
\[\max_{j\in\{1,2,\cdots ,r\}}(\sup_{t\in[0,a]}(||W^{*}S^{'*}_{t}W-e^{-i\lambda t}1_{\cll})g_{j}||))=||(S_{t}^{*}-e^{-i\lambda t}1_{L^{2}(\bbr_{+})})g||<\epsilon\] Since $S^{'}$ is unitarily equivalent to $\{S_{t}\otimes 1_{\clk}\}_{t\geq 0}$, the proposition is proved.
\end{proof}

\section{Pathologies when $d=\infty$}
In this section, we indicate the pathologies that occur when $d=\infty$. Let \[
\mathbb{R}^{\infty}:=\{(x_n) \in \bbr^\bbn: x_n=0 \textrm{~eventually}\}.
\]
We consider $\mathbb{R}^{\infty}$ as the inductive limit $\displaystyle \lim_{n}(\bbr^n,\iota_n)$, where the connecting map $\iota_n:\bbr^n \to \bbr^{n+1}$ is given by $(x_1,x_2,\cdots,x_n) \to (x_1,x_2,\cdots,x_n,0)$. Recall that a subset $K \subset \bbr^{\infty}$ is compact if and only if there exists $n \geq 1$ such that $K$ is a compact subset of $\bbr^{n}$. After an abuse of notation, we view $\bbr^n$ as a subset of $\bbr^\infty$.
Let $\bbr_{+}^{\infty}:=\{(x_n) \in \bbr^{\infty}:x_n \geq 0\}$. Note that $\bbr_{+}^{\infty}$ is a closed subsemigroup of $\bbr_{+}^{\infty}$. Let $e_i=(0,0,\cdots,1,0,\cdots)$, where $1$ occurs at the $i$-th position. 

As in the finite case ($d<\infty$), an isometric representation $V=\{V_{\underline{t}}\}_{\underline{t} \in \bbr_{+}^{\infty}}$ (which is always assumed to be strongly continuous) is said to be \textbf{doubly commuting} if \[V_{te_i}^{*}V_{se_j}=V_{se_j}V_{te_i}^{*}\] for every $s,t \geq 0$ and $i \neq j$. For an isometric representation $V=\{V_{\underline{t}}\}_{\underline{t} \in \bbr_{+}^{\infty}}$ and for $i \in \bbn$, we set 
\[
V_{t}^{(i)}=V_{te_i}
\]
for $t\geq 0$. 

We show some pathological properties concerning the structure of isometric representations, even the doubly commuting ones, of $\bbr_{+}^{\infty}$. In particular, we show that $\mathcal{R}_0(\bbr_{+}^{\infty})$ is not $T_0$ and also prove that  Wold decomposition fails. First, we show that $\mathcal{R}(\bbr_+^{\infty})$ is not $T_0$  by exhibiting two distinct points in $\clr_{0}(\rp^{\infty})$ which have the same closure in $\clr_{0}(\rp^{\infty})$.   In the process, we also construct a continuum of strongly pure, irreducible, doubly commuting isometric representations of $\bbr_{+}^{\infty}$. This is in contrast to the finite case, where Cooper's theorem states that there is only one strongly pure, irreducible, doubly commuting isometric representation

 Let us first fix some notation. 
 Let $\gamma$ be the standard Gaussian distribution on $\bbr$, i.e. $\gamma$ is the probability measure on $\bbr$ given by
\[\gamma(E)=\frac{1}{\sqrt{2\pi}}\int_{E}e^{\frac{-x^{2}}{2}}dx.\] Let $\mu$ be the product measure $\displaystyle \otimes_{n=1}^{\infty}\gamma$ on $\bbr^\bbn$. 
For $\underline{x} \in \bbr^n$, the map $\bbr^n \ni \underline{y} \to \underline{y}+\underline{x} \in \bbr^n$ will be denoted by $T_{\underline{x}}^{(n)}$. For $\underline{x} \in \bbr^\bbn$, let $T_{\underline{x}}:\bbr^\bbn \to \bbr^\bbn$ be defined by 
\[
T_{\underline{x}}(\underline{y})=\underline{x}+\underline{y}.
\]
In what follows, we consider $\bbr^\bbn$ as a measure space endowed with the probability measure $\mu=\otimes_{k=1}^{\infty}\gamma$.  The Lebesgue measure on $\bbr^n$ will be denoted by the same letter $\lambda$ for every $n$. For $n \geq 1$, let $\mu^{(n)}=\bigotimes_{k=1}^{n}\gamma$ which is a probability measure on $\bbr^n$. 

\begin{remark}
For $\underline{x} \in \bbr^\bbn$, the measure $\mu$ is quasi-invariant for $T_{\underline{x}}$, i.e. $\mu \circ T_{\underline{x}}^{-1}$ and $\mu$ are absolutely continuous w.r.t. each other if and only if $\underline{x} \in \ell^2(\bbn)$. This is well known and can be proved by applying Prop. III-1-2. of \cite{Neveu}.
\end{remark}

Let $X$ denote the set of all sequences $(a_{n})_{n\in\bbn}\in \bbr^{\bbn}$ such that $a_{n}\leq 0$ for each $n\in\bbn$ and
\[{\prod}_{n=1}^{\infty}\frac{1}{\sqrt{2\pi}}\int_{a_{n}}^{\infty}e^{\frac{-x^{2}}{2}}dx>0.\]
\begin{lemma} The set $X$ is non-empty.
\end{lemma}
\begin{proof}It suffices to exhibit a sequence $a=(a_{n})\in\bbr^{\bbn}$ with $a_{n}<0$ for each $n\in\bbn$ such that
\[\sum_{n=1}^{\infty}\log\Big(\frac{1}{\sqrt{2\pi}}\int_{a_{n}}^{\infty} e^{\frac{-x^{2}}{2}}dx\Big)\] converges.

Note that if $t<0$, then $\frac{1}{2}<\frac{1}{\sqrt{2\pi}}\int_{t}^{\infty} e^{\frac{-x^{2}}{2}}dx<1$. Let $-\log2<t<0$. For each $n\in\bbn$, let $t_{n}:=e^{t^{2n+1}}$. Then, $\frac{1}{2}<t_{n}<1$, and $\sum_{n=1}^{\infty}\log(t_{n})$ converges.
For $n\in\bbn$, choose $a_{n}<0$ such that \[\frac{1}{2}<\frac{1}{\sqrt{2\pi}}\int_{a_{n}}^{\infty} e^{\frac{-x^{2}}{2}}dx=t_{n}.\]
Such a choice is possible since the function
\[s \to \frac{1}{\sqrt{2\pi}}\int_{s}^{\infty} e^{\frac{-x^{2}}{2}}dx \] is continuous. Now, it is clear that the series
\[\sum_{n=1}^{\infty}\log(\frac{1}{\sqrt{2\pi}}\int_{a_{n}}^{\infty} e^{\frac{-x^{2}}{2}}dx)=\sum_{n=1}^{\infty} t_{n}\] converges. This completes the proof.
\end{proof}

Let $a:=(a_1,a_2,\cdots) \in X$. Let \[A:=\prod_{n=1}^{\infty}[a_n,\infty)=\{(x_{n})\in\rn: \textrm{ $x_{n}\geq a_{n}$}\}.\] Then, $A$ has positive measure.  We consider $L^2(A)$ as a subspace of $L^2(\bbr^\bbn)$. For $n \geq 1$, we let 
\begin{align*}
    A_{n]}:&=\{\underline{x} \in \bbr^n: x_i \geq a_i \textrm{~for all $i$}\}, \textrm{~and} \\
    A_{(n}&=\{\underline{x} \in \bbr^\bbn: x_1 \geq a_{n+1},x_2 \geq a_{n+1},\cdots\}.
\end{align*}

For $\underline{x}=(x_{1},x_{2},\cdots , x_{k},\cdots)\in\rpinf$, define $T_{\underline{x}}: \bbr^{\bbn}\to \bbr^{\bbn}$ by
\[T_{\underline{x}}(y_{1},y_{2},\cdots, y_{k},y_{k+1},\cdots ):= (y_{1}+x_{1},y_{2}+{x_{2}},\cdots, y_{k}+x_{k},\cdots ).\]
Note that $T_{\underline{x}}$ is a measurable, invertible map and the measure $\mu$ is quasi-invariant for $T_{\underline{x}}$. 
Consider the Koopman operator $U_{\underline{x}}:L^{2}(\bbr^{\bbn})\to L^{2}(\bbr^{\bbn}) $  defined by
\[U_{\underline{x}}f=\sqrt{\frac{d(\mu \circ T_{\underline{x}}^{-1})}{d\mu}}f\circ T_{\underline{x}}^{-1}.\]

Since $A+\bbr_{+}^{\infty} \subset A$, $U_{\underline{x}}$ leaves $L^2(A)$ invariant for every $\underline{x} \in \bbr_{+}^{\infty}$. For $\underline{x} \in \bbr_{+}^{\infty}$,
let $V^{A}_{\underline{x}}$ denote the restriction of $U_{\underline{x}}$ to $L^{2}(A)$.
Then, $V^{A}:=\{V^{A}_{\underline{x}}\}_{\underline{x}\in\rpinf}$ is a doubly commuting isometric representation of $\rpinf$ on $L^{2}(A)$. Note that $V^{A}$ is strongly pure, i.e. for every $i$, $V_{te_i}^{A*} \to 0$ as $t\to \infty$ in SOT. 
For $\underline{x}=(x_{1}, x_{2}, \cdots , x_{k},0,0,\cdots)\in\rpinf$, $V^{A}_{\underline{x}}$ is given by
\[V^{A}_{\underline{x}}f(\underline{y})= \begin{cases}\sqrt{\displaystyle{\prod}_{n=1}^{k}e^{\frac{2x_{k}y_{k}-x_{k}^{2}}{2}}}f(y_{1}-x_{1}, y_{2}-x_{2},\cdots, y_{k}-x_{k}, y_{k+1}, \cdots), & \textrm{ if $\underline{y}-\underline{x}\in A$},\\
0, & \textrm{otherwise.}
\end{cases}\]
\begin{prop}
\label{irred_infinity}With the foregoing notation, the isometric representation $V^{A}$ is irreducible.
    
\end{prop}
\begin{proof}
We denote $V^{A}_{\underline{t}}$ by $V_{\underline{t}}$. Denote the commutant of $\{V_{\ut}, V_{\ut}\}_{\ut\in\rpn}$ by $\mathcal{M}(V^{A})$.  For $\phi \in L^{\infty}(A)$, let $M_\phi$ be the multiplication operator on $L^2(A)$ defined by 
\[
M_{\phi}(\xi)=\phi \xi.
\]
Suppose $T\in \mathcal{M}(V^A)$.

\textit{Claim:} The operator $T$ commutes with $M_{\phi}$ for every $\phi \in L^{\infty}(A)$.

Note that the $\sigma$-algebra of  $\bbr^\bbn$ is generated by cylindrical sets of the form $\prod_{n\in\bbn} X_{n}$, where each $X_{n}$ is an interval in $\bbr$, and $X_{n}=\bbr$ except for finitely many $n\in\bbn$. Thus, it suffices to prove that $T$ commutes with the multiplication operators corresponding to the indicator functions of these cylindrical sets.

 Consider a cylindrical set of $L^{2}(\bbr^\bbn)$ of the form $X:=\prod_{n\in\bbn}X_{n}$,  where each $X_{n}$ is an interval in $\bbr$ and $X_{n}=\bbr$ except for finitely many $n$.  Let $k$ be such that $X_n=\bbr$ for $n \geq k+1$. 
  Let $M_{X}$ denote the multiplication operator corresponding to the indicator function of $X$. 

Let $\displaystyle \mu^{(k)}=\otimes_{n=1}^{k}\mu$. We may identify $L^{2}(\bbr_{+}^{k}, \mu^{(k)})\otimes L^{2}(A_{(k},\mu)$ with $L^{2}(A,\mu)$ via the unitary $U:L^{2}(\bbr_{+}^{k},\mu^{(k)})\otimes L^{2}(A_{(k},\mu)\to  L^{2}(A,\mu)$ defined by 
\[
U(f \otimes g)(x_1,x_2,\cdots)=f(x_1-a_1,x_2-a_2,\cdots,x_n-a_k)g(x_{k+1},x_{k+2},\cdots).
\]
Let $W:L^{2}(\bbr_{+}^k, \mu^{(k)}) \to L^{2}(\bbr_{+}^k,\lambda)$ be the unitary defined by 
\[
Wf=\sqrt{\frac{d\mu^{(k)}}{d\lambda}}f.
\]
Set $W_0=(W \otimes 1)U^{*}$. 
Then, for each $\ut\in\rp^{k}$, 
\[W_0V^{A}_{\ut}W_0^{*}=S_{\ut}\otimes 1.\]
Since $\{S_{\ut}\}_{\ut\in\rp^{k}}$ is an irreducible isometric representation (Lemma \ref{commutant}), \begin{equation}
    \label{tail}
    W_0TW_0^*=1\otimes T_{k}
\end{equation} for some $T_{k}\in B(L^{2}(A_{(k},\mu))$. Note that 
\begin{equation}
    \label{initial}
    W_0M_XW_0^{*}=M_{\phi}\otimes 1
\end{equation}
for some $\phi \in L^{\infty}(\bbr_{+}^k)$. 
It follows from Eq. \ref{tail} and Eq. \ref{initial} that  $W_0TW_0^*$ and $W_0M_XW_0^*$ commutes. Hence, $T$ commutes with $M_X$. This proves the claim.

Since $T$ commutes with the multiplication operator $M_{f}$ for every $f\in L^{\infty}(A)$, $T$ must itself be a multiplication operator $M_g$ corresponding to some $g\in L^{\infty}(A)$.
We now prove that $g$ is constant almost everywhere. For $n \in \bbn$, let $\pi_n:A \to \bbr$ be the projection onto the $n$-th coordinate.  
For $n\in\bbn$, let $\clg_{n}$ be the smallest (complete) $\sigma$- algebra  which makes $\pi_{n},\pi_{n+1},\cdots$ measurable. Let \[\mathcal{F}_{\infty}:=\displaystyle \bigcap_{n\in \bbn} \clg_{n}.\] 
 
It follows from Eq. \ref{tail} that given $k \in \bbr$, there exists $g_k \in L^{\infty}(A_{(k})$ such that for almost all $\underline{x} \in \prod_{i=1}^{\infty}[a_i,\infty)$, \[
g(x_1,x_2,\cdots)=g_k(x_{k+1},x_{k+2},\cdots).
\] 
Thus, $g$ is $\mathcal{F}_\infty$-measurable.  By Kolmogorov's zero-one law, $g$ must be constant almost everywhere. In other words, $T=\lambda$ for some $\lambda\in\bbc$. The proof is over.
\end{proof}

Let $a=(a_{n})_{n\in\bbn}, b=(b_{n})_{n\in\bbn} \in X$. Let $A:=\displaystyle \prod_{n=1}^{\infty}[a_n,\infty)$ and $B:=\displaystyle \prod_{n=1}^{\infty}[b_n,\infty)$.
\begin{prop} 
\label{continuum of irrep} The isometric representations $V^{A}$ and $V^{B}$ are unitarily equivalent if and only if $(a_{n}-b_{n})_{n\in\bbn}\in l^{2}(\bbn)$.
\end{prop}
\begin{proof} Suppose $\underline{c}:=(b_{n}-a_{n})\in l^{2}(\bbn)$. Then, $\mu$ is quasi-invarinat under the translation $T_{\underline{c}}$. Define the Koopman operator $U:L^{2}(\bbr^\bbn) \to L^{2}(\bbr^\bbn)$ by 
\[
Uf=\sqrt{\frac{d(\mu \circ T_{\underline{c}}^{-1})}{d\mu}}f \circ T_{\underline{c}}^{-1}.\]
Note that $U$ is a unitary operator, $U$ maps $L^2(A)$ onto $L^2(B)$ and intertwines $V^A$ and $V^B$. 

Conversely, suppose that $V^A$ and $V^B$ are unitarily equivalent. Let $U:L^2(A) \to L^2(B)$ be a unitary such that $UV^A_{\ut}U^*=V^B_{\ut}$ for every $\ut \in \bbr_{+}^{\infty}$. 

\textit{Claim:} For every $\phi \in L^{\infty}(A)$, \begin{equation}
    \label{twist}
    UM_{\phi}U^{*}=M_{\phi \circ T_{-\underline{c}}}.
\end{equation}
For $n \geq 1$, we identify $L^{\infty}(A_{n]})$ as a subalgebra of $L^{\infty}(A)$ via the map \[L^{\infty}(A_{n]}) \ni \phi \to \phi \circ \pi_{n]} \to L^{\infty}(A).\] Here, $\pi_{n]}$ is the projection onto the first $n$-coordinates. 
Note that $\bigcup_{n=1}^{\infty}L^{\infty}(A_{n]})$ is weak $*$-dense in $L^{\infty}(A)$. Thus, it suffices to verify Eq. \ref{twist} when $\phi \in L^{\infty}(A_{n]})$ for some $n$. 

Let $n \geq 1$ and let $\phi \in L^{\infty}(A_{n]})$ be given. Let $a^{(n)}=(a_1,a_2,\cdots,a_n)$, $b^{(n)}=(b_1,b_2,\cdots,b_n)$ and $c^{(n)}=a^{(n)}-b^{(n)}$. Let $U_1:L^{2}(A_{n]},\mu^{(n)}) \to L^{2}([0,\infty)^{n},\mu^{(n)})$ be the unitary defined by $U_1f(\underline{x})=f(\underline{x}+a^{(n)})$. Let $W:L^{2}([0,\infty)^{n},\mu^{(n)}) \to L^{2}([0,\infty)^{n},\lambda)$ be the unitary defined by \[Wf=\sqrt{\frac{d\mu^{(n)}}{d\lambda}}f.\] Define another unitary $U_2:L^{2}(B_{n]},\mu) \to L^2([0,\infty)^{n},\mu^{(n)})$ by $U_2(f)(\underline{x})=f(\underline{x}+b^{(n)})$. Let $U_3:L^{2}(A_{(n},\mu) \to L^2(B_{(n},\mu)$ be a unitary operator. 

Note that for $\ut \in \bbr_{+}^{n}$, \[
(WU_1 \otimes U_3)V^{A}_{\ut}(WU_1 \otimes U_3)^{*}=S_{\ut} \otimes 1\]
and \[
(WU_2 \otimes 1)UV_{\ut}^{A}U^*(WU_2 \otimes 1)^{*}=(WU_2 \otimes 1)V_{\ut}^{B}(WU_2 \otimes U_3)^{*}=S_{\ut} \otimes 1.\]

Hence, $W_0:=(WU_1 \otimes U_3)U^{*}(WU_2 \otimes 1)^{*}$ lies in the commutant of $\{S_{\ut} \otimes 1, S_{\ut}^{*} \otimes 1: \ut \in \bbr_{+}^{n}\}$. Hence, $W_0$ is of the form $1 \otimes W_1$ for some unitary operator $W_1:L^2(A_{(n}) \to L^2(B_{(n})$.  This implies $U=U_2^*U_1 \otimes W_1^*U_3$. Calculate as follows to observe that 
\begin{align*}
UM_{\phi}U^{*}&=(U_2^*U_1 \otimes W_1^*U_3)(M_\phi \otimes 1)(U_1U_2^* \otimes U_3^*W_1)\\ 
&=M_{\phi \circ T_{-c^{(n)}}^{(n)}} \otimes 1\\
&=M_{\phi \circ T_{-\underline{c}}}.
\end{align*}
Therefore, Eq. \ref{twist} holds for every $\phi \in \bigcup_{n=1}^{\infty} L^{\infty}(A_{n]})$, and hence it holds for every $\phi \in L^{\infty}(A)$. 

Thus, for every Borel subset $E \subset A$, $UM_{1_{E}}U^{*}=M_{1_{E+\underline{c}}}$. This implies that $\mu(E)=0$ if and only if $\mu(E+\underline{c})=0$. Thus, the push-forward measure $\mu \circ T_{\underline{c}}^{-1}$ and $\mu$ are equivalent. 

It follows from Prop. III-1-2. of \cite{Neveu} that $I:=\displaystyle \lim_{n \to \infty}\int_{B_{n]}}\sqrt{\frac{d(\mu^{(n)} \circ T_{c^{(n)}}^{-1})}{d\mu}} d\mu >0$. Note that \[
\frac{d(\mu^{(n)}\circ T_{c^{(n)}}^{-1})}{d\mu^{(n)}}(\underline{x})=e^{\langle c^{(n)}|\underline{x}\rangle}e^{-\frac{||c^{(n)}||^{2}}{2}}.\]
A routine computation of the integrals involving Gaussian distributions shows that \[
I=\exp\Big(-\frac{\sum_{n=1}^{\infty}c_n^2}{8}\Big)\prod_{n=1}^{\infty}\int_{b_n}^{\infty}e^{-\frac{1}{2}(t+\frac{c_n}{2})^2}. 
\]
Hence, $\underline{c} \in \ell^2(\bbn)$. This completes the proof.
\end{proof}

\begin{theorem}
Let $\clh$ be an infinite dimensional Hilbert space. Let $\underline{a},\underline{b} \in X$ be such that $\underline{a}-\underline{b} \notin \ell^2(\bbn)$. Set $A:=\prod_{n=1}^{\infty}[a_n,\infty)$ and $B:=\prod_{n=1}^{\infty}[b_n,\infty)$. 
 The following statements are true:
\begin{enumerate}[(1)]
\item The singletons $\{[V^{A}\otimes 1_{\clh}]\}$ and $\{[V^{B}\otimes 1_{\clh}]\}$ have the same closure in $\clr_{0}(\rpinf)$, and
\item $\clr_{0}(\rpinf)$ is not a $T_{0}$- space.\end{enumerate}
\end{theorem} 
\begin{proof} 
\begin{enumerate}
     \item  Let $\epsilon>0$, $K\subset\rpinf$ be compact. 
    We recall that a compact subset of $\rpinf$ is necessarily finite dimensional, i.e. there exists $n\in\bbn$ such that $K\subset \bbr^{n}_{+}$. We may identify $L^2(A,\mu)$ with $L^{2}(A_{n]},\mu^{(n)})\otimes L^2(\displaystyle A_{(n},\mu)$. Let $a^{(n)}=(a_1,a_2,\cdots,a_n)$, and let $T_{a^{(n)}}:\bbr^n \to \bbr^n$ be the translation given by \[
    T_{a^{(n)}}(\underline{x})=a^{(n)}+\underline{x}.\]  Let 
    $U_{a^{(n)}}:L^{2}(\bbr^n,\mu^{(n)}) \to L^{2}(\bbr^n,\mu^{(n)})$ be the Koopman operator associated with $T_{-a^{(n)}}$. Then, $U_{a^{(n)}} \otimes 1$ maps $L^2(A)$ onto $L^{2}([0,\infty)^{n}\times A_{(n},\mu^{(n)}\otimes \mu)$, and we identify the latter Hilbert space with $L^{2}([0,\infty)^{n},\mu^{(n)})\otimes L^{2}(A_{(n},\mu)$. 
    
    Let $W:L^{2}([0,\infty)^{n},\mu^{(n)}) \to L^{2}([0,\infty)^{n},\lambda)$ be the unitary defined by \[Wf=\sqrt{\frac{d\mu^{(n)}}{d\lambda}}f.\] Set $W_0=(W \otimes 1)(U_{a^{(n)}}\otimes 1)$. A routine verification implies that for $\underline{t} \in \mathbb{R}_{+}^{n}$, 
    \[
    (W_0 \otimes 1)(V_{\underline{t}}^A \otimes 1_{\clh})(W_0 \otimes 1)^{*}=S_{\underline{t}} \otimes 1_{\cll_0},
    \]
    where $\cll_0=L^{2}(A_{(k},\mu)\otimes \clh$. After identifying $\cll_0$ with $\clh$ via a unitary, we see that there exists a unitary operator $W_1$ such that, for $\ut \in \bbr_{+}^{n}$, 
    \[
    W_1(V^A_{\ut}\otimes 1_{\clh})W_1^*=S_{\ut}\otimes 1_{\clh}.
    \]
    A similar argument shows that there exists a unitary $W_2$ such that, for $\ut \in \bbr_{+}^{n}$, 
    \[
    W_2(V^B_{\ut}\otimes 1_{\clh})W_{2}^{*}=S_{\ut} \otimes 1_{\clh}.
    \]
    Setting $W:=W_{1}^{*}W_2$, we see that 
    \[
    W(V^B_{\ut} \otimes 1_{\clh})W^*=V^A_{\ut}\otimes 1_{\clh}\]
    for every $\ut \in K$. 
     It is now clear that every basic open set containing $[V^{A}\otimes 1_{\clh}]$ also contains $[V^{B}\otimes 1_{\clh}]$ and vice versa. In other words, the closures of $\{[V^{A}\otimes 1_{\clh}]\}$ and $\{[V^{B}\otimes 1_{\clh}]\}$ in $\clr_{0}(\rpinf)$ are the same.
\item Thanks to Prop. \ref{irred_infinity} and Prop. \ref{continuum of irrep}, $V^{A}\otimes 1_{\clh}$ and $V^B \otimes 1_{\clh}$ are railway representations that are inequivalent. By (1), 
 $[V^{A}\otimes 1_{\clh}]$ and $[V^{B}\otimes 1_{\clh}]$ cannot be separated by open sets, and hence, it follows that $\clr_{0}(\rpinf)$ is not $T_0$.
\end{enumerate}
\end{proof}
We finish our paper by showing that Wold decomposition fails for doubly commuting isometric representations of  $\bbr_{+}^{\infty}$.
\begin{definition}
    Let $V=\{V_{\underline{t}}\}_{\underline{t} \in \mathbb{R}_+^{\infty} }$ be a strongly continuous, isometric representation of  the semigroup $\mathbb{R}_+^{\infty}$ on a Hilbert space $\mathcal{H}.$ The  representation $V$ is said to admit a Wold decomposition if the Hilbert space $\mathcal{H}$ decomposes as an orthogonal direct sum:
   \[
\mathcal{H} = \bigoplus_{g \in \{0,1\}^{\mathbb{N}}} \mathcal{H}_g,
\]
such that  each  \( \mathcal{H}_g \) is a reducing  subspace  for  $V$, and for each   \(i\in \mathbb{N},  \)  the restriction of the representation $ \{V^{(i)}_t\}_{t \in \bbr_+}$  to $ {\mathcal{H}_g} $  satisfies:

\begin{itemize}
  \item    if \( g(i) = 0 \), then $ \{V^{(i)}_t\}_{t \in \bbr_+}$   on $ {\mathcal{H}_g} $  is a \textbf{unitary} representation,
  \item if \( g(i) = 1 \), then  $ \{V^{(i)}_t\}_{t \in \bbr_+}$  on  $ {\mathcal{H}_g} $ is a \textbf{pure isometric} representation. 
\end{itemize}
\end{definition}

A similar definition applies to isometric representations $V=\{V_{\underline{n}}\}_{\underline{n} \in \mathbb{N}_0^{\infty} }$  of the semigroup $\mathbb{N}_0^{\infty}=\{(m_i): m_i \in \bbn_0, m_i=0 \textrm{~~eventually}\}$.

\begin{prop}
\label{failure of Wold}
There exists a  strongly continuous, doubly commuting isometric representation of the semigroup  $\bbr_{+}^{\infty}$ on a Hilbert space which does not admit a Wold decomposition. 
\end{prop}
\begin{proof}
  Let $\nu$ be a probability measure on  the set $\{0,1\},$ and consider the infinite   product measure $\mu:=\displaystyle \otimes_{n=1}^{\infty} \nu$ on   $\{0,1\}^{\mathbb{N}}.$  For  each $x \in \{0,1\}^{\mathbb{N}},$ define $\clh(x)=\clh$, where $\mathcal{H}$ is a fixed Hilbert space.  The direct integral of  the  family of Hilbert spaces $\{ \mathcal{H}(x)\}_{x \in \{0,1\}^{\mathbb{N}}}$  with respect to  the measure $ \mu$  is $$\int^{\oplus}H(x)d\mu(x)=L^2(\{0,1\}^{\mathbb{N}}, \mathcal{H}).$$

 Let $S=\{S_{\underline{t}}\}_{\underline{t} \in \mathbb{R}_+^{\infty}}$  be a strongly continuous, strongly pure, doubly commuting isometric representation of $\mathbb{R}_+^{\infty}$ on the Hilbert space $\mathcal{H}$, and let $U=\{U_{\underline{t}}\}_{\underline{t} \in \mathbb{R}_+^{\infty}}$ be a  strongly continuous, unitary  representation of $\mathbb{R}_+^{\infty}$ on  $\mathcal{H}$ such that $U$  doubly commutes with $S.$ This means that for all $s, t \in \mathbb{R}_+,$
$$S_t^{(i)}U_s^{(j)}=U_s^{(j)}S_t^{(i)}  \,\, \textrm{and}\,\,S_t^{(i)*}U_s^{(j)}=U_s^{(j)}S_t^{(i)*}, \hspace{0.8cm} i,j \in \mathbb{N}.$$

Let  $x=(x_i) \in \{0,1\}^{\mathbb{N}}$,  $i \in \mathbb{N}$ and $t \in \bbr_+$. Define an isometry  on $\mathcal{H}(x)$ by
$$ V_{t}^{(i)}(x)=\begin{cases}S_t^{(i)}, & \textrm{ if $x_i=1$,}\\
U_t^{(i)}, & \textrm{if $x_i=0$.} 
\end{cases}$$
 
 Since $ V_{t}^{(i)}(x)$ is an isometry on each fiber  $\mathcal{H}(x),$ the decomposable operator $\int^{\oplus }V_{t}^{(i)}(x)d\mu(x)$  of the $\mu$-measurable operator valued function  $x \to V_{t}^{(i)}(x)$ is an isometry on $L^2(\{0,1\}^{\mathbb{N}}, \mathcal{H})$. We denote it  by $V_{t}^{(i)}.$  Since  $U$  doubly commutes with S,   the resulting semigroup   $V=\{V_{\underline{t}}\}_{\underline{t} \in \mathbb{R}_+^{\infty}} $  constructed from $\{ V_{t}^{(i)}\}_{t \in \bbr_+},  i \in \mathbb{N}$  defines a strongly continuous, doubly commuting isometric representation   of $\mathbb{R}_+^{\infty}$ on the Hilbert space  $L^2(\{0,1\}^{\mathbb{N}}, \mathcal{H}).$ 
 
 For any  $d \in \mathbb{N},$    the restricted family  $V^d:=\{V_{\underline{t}}\}_{\underline{t} \in \bbr_+^d}$ defines a strongly continuous, doubly commuting isometric representation of $\mathbb{R}_+^{d}$ on the Hilbert space   $L^2(\{0,1\}^{\mathbb{N}}, \mathcal{H})$.
 
 Applying  the  Wold decomposition theorem (Thm. \ref{WDCV}) to the representation $V^d$ on  the Hilbert space $L^2(\{0,1\}^{\mathbb{N}}, \mathcal{H})$,  we get an orthogonal  decomposition 
 \begin{align}\label{FWD} L^2(\{0,1\}^{\mathbb{N}}, \mathcal{H})= \bigoplus_{g \in \{0,1\}^d } L^2(X_g, \mathcal{H}),\end{align} where $X_g= \{ x \in \{0,1\}^{\mathbb{N}}\, : \, x_i=g(i), i \in I_d \}$ such that  for each  $g \in \{0,1\}^d:$ 
 \begin{enumerate}
 \item $L^2(X_g, \mathcal{H})$ is a  reducing subspace for the representation  $V^d.$ 
 \item For each $i \in I_d,$  the isometric representation   $\{V_{t}^{(i)}\}_{t \in \bbr_+}$  restricted to $L^2(X_g, \mathcal{H})$ is pure if $g(i)=1$     and unitary if $g(i)=0.$
 \end{enumerate}
  Furthermore,  for  every $ i \in I_d$ and $t \in \bbr_+,$ the isometry $V_{t}^{(i)}$ can be expressed as
 \begin{align}V_{t}^{(i)}= \bigoplus _{g \in \{0,1\}^d}V_{t}^{(i)}|_{L^2(X_g, \mathcal{H})}=\bigoplus _{g \in \{0,1\}^d}\int^{\oplus }_{ X_g }V_{t}^{(i)}(x)d\mu(x).\end{align}

Suppose that  the  strongly continuous, doubly commuting isometric representation $V$ over $\mathbb{R}_+^{\infty}$ admits a Wold decomposition, meaning that the Hilbert space $L^2(\{0,1\}^{\mathbb{N}}, \mathcal{H})$ decomposes as
\begin{align}\label{WDI} L^2(\{0,1\}^{\mathbb{N}}, \mathcal{H})= \bigoplus_{g \in \{0,1\}^{\mathbb{N}} } \mathcal{H}_g, \end{align}
where each  $\mathcal{H}_g$ is a reducing subspace of $V$ such that  the representation $\{V_{t}^{(i)}\}_{t \in \bbr_+}$  restricted to  ${\mathcal{H}_g}$ is pure  if $g(i)=1$   and unitary if $g(i)=0.$  

Now,  for $g \in \{0,1\}^{\mathbb{N}}$  and $d \in \mathbb{N}$, define $g_d$ as the restriction of $g$ to the index set  ${I_d} ,$ i.e.   $g_d=g|_{I_d}.$ By the  uniqueness of   Wold decomposition  (Thm. \ref{WDCV})  for  the  doubly commuting isometric representation $V^d$
on $L^2(\{0,1\}^{\mathbb{N}}),$ if follows that  $\mathcal{H}_g \subset  L^2(X_{g_d}, \mathcal{H})$ for all $d \in \mathbb{N}$, where $X_{g_d}$  is defined as in Eq. \ref{FWD}. Since $d \in \mathbb{N}$ is arbitrary,  we have  $\mathcal{H}_g \subset  L^2(X_{g}, \mathcal{H}),$
  where $$X_g= \{x= (x_i) \in \{0,1\}^{\mathbb{N}}\, :\, g(i)=x_i\}=\{(g(i))_{i \in \mathbb{N}}\}.$$   Observe  that $\mu (\{(g(i))_{ i \in \mathbb{N}}\})=0,$ which implies that $L^2(X_g, \mathcal{H})=0,$ and therefore $\mathcal{H}_g=\{0\}.$ This leads to the conclusion that $L^2(\{0,1\}^{\bbn}, \mathcal{H})=\{0\},$ which is a contradiction. Hence, $V$ does not admit a Wold decomposition. 
  \end{proof} 

  \begin{remark}
  Let $\bbn_0^{\infty}=\{(n_i) \in \bbn_0^\bbn: \textrm{$n_i=0$ except for finitely many $i$}\}$.
        Let $S=\{S_{\underline{n}}\}_{\underline{n} \in \mathbb{N}^{\infty}}$  and  $U=\{U_{\underline{n}}\}_{\underline{n} \in \mathbb{N}^{\infty}}$ be  doubly commuting isometric representations of $\mathbb{N}^{\infty}_0$ on the Hilbert space $\mathcal{H}$  such that $S$ is strongly pure, $U$ is unitary and  $U$  doubly commutes with $ S.$ 
  
  For $x=(x_n) \in \{0,1\}^{\mathbb{N}}$  and $n \in \mathbb{N},$ define an isometry  on $\mathcal{H}(x)$ by
$$ V_{i}(x)=\begin{cases}S_i, & \textrm{ if $x_i=1$,}\\
U_i, & \textrm{if $x_i=0$.} 
\end{cases}$$
 Since $ V_{i}(x)$ is an isometry on $\mathcal{H}(x),$ the decomposable operator $\int^{\oplus }V_i(x)d\mu(x)$  of the $\mu$-measurable operator valued function  $x \to V_{i}(x)$ is an isometry on $L^2(\{0,1\}^{\mathbb{N}}, \mathcal{H})$ and it is denoted by $V_i.$ For $\underline{n} \in \bbn_0^{\infty}$, set  $V_{\underline{n}}=\prod_{i=1}^{\infty}V_i^{n_i}$. As $U$  doubly commutes with S, it follows that the  isometric representation $V=\{V_{\underline{n}}\}_{\underline{n} \in \mathbb{N}^{\infty}}$  of $\mathbb{N}^{\infty}$  on the Hilbert space  $L^2(\{0,1\}^{\mathbb{N}}, \mathcal{H})$ is  also doubly commuting. 

Arguing similarly as in Prop. \ref{failure of Wold}, together with the discrete analogue of Thm. \ref{WDCV}, we can show that the doubly commuting isometric representation  $V=\{V_{\underline{n}}\}_{\underline{n} \in \mathbb{N}^{\infty}}$ of $\mathbb{N}^{\infty}$ on $L^2(\{0,1\}^{\mathbb{N}}, \mathcal{H}),$ as defined above,  does not admit a Wold decomposition. 

 \end{remark}

\end{document}